\documentclass[11pt]{amsart}
\usepackage[margin=1.5in]{geometry}
\usepackage{amsmath, amsthm, amssymb, amsfonts}

\usepackage[pagebackref=true]{hyperref} 
\renewcommand*\backref[1]{\ifx#1\relax \else (Cited on #1) \fi}

\usepackage[alphabetic]{amsrefs}
\usepackage{graphicx}
\usepackage{mathtools}
\usepackage{mathrsfs}
\usepackage{tikz-cd}
\usepackage{comment}
\usepackage{cite}

\usepackage{enumerate}

\numberwithin{equation}{section}

\newtheorem{thm}{Theorem}
\newtheorem{theorem}[thm]{Theorem}

\newtheorem{lem}{Lemma}[section]               
\newtheorem{lemma}[lem]{Lemma}               

\newtheorem{cor}[lem]{Corollary}
\newtheorem{corollary}[lem]{Corollary}               
\newtheorem{prop}[lem]{Proposition}

\theoremstyle{definition}

\newtheorem{defn}[lem]{Definition}

\theoremstyle{remark}

\newtheorem{remark}[lem]{Remark}

\numberwithin{figure}{section}
\numberwithin{table}{section}

\newcommand{\R}{\mathbb{R}}
\newcommand{\Z}{\mathbb{Z}}
\newcommand{\F}{\mathbb{F}}

\newcommand{\T}{\mathbb{T}}
\newcommand{\N}{\mathbb{N}}

\newcommand{\ra}{\rightarrow}

\newcommand{\sub}{\subset}

\DeclareMathOperator{\id}{id}

\DeclareMathOperator{\Ord}{Ord}

\newcommand{\al}{\alpha}
\newcommand{\be}{\beta}

\newcommand{\bal}{\pmb{\alpha}}
\newcommand{\bbe}{\pmb{\beta}}
\newcommand{\bga}{\pmb{\gamma}}

\newcommand{\cA}{\mathcal{A}}

\newcommand{\cD}{\mathcal{D}}

\newcommand{\cF}{\mathcal{F}}
\newcommand{\cG}{\mathcal{G}}
\newcommand{\cH}{\mathcal{H}}

\newcommand{\frs}{\mathfrak{s}}
\newcommand{\frt}{\mathfrak{t}}

\newcommand{\w}{\textbf{w}}
\newcommand{\x}{\textbf{x}}
\newcommand{\y}{\textbf{y}}
\newcommand{\z}{\textbf{z}}
\renewcommand{\L}{\mathbb{L}}
\newcommand{\U}{\mathbb{U}}

\newcommand{\SpinC}{\text{Spin}^{\text{C}}}
\newcommand{\tors}{\text{Tors}}
\newcommand{\pt}{\{\text{pt}\}}
\newcommand{\Sym}{\text{Sym}}

\newcommand{\cCFL}{\mathcal{C\!F\!L}}
\newcommand{\cHFL}{\mathcal{H\!F\! L}}
\newcommand{\CF}{\mathit{CF}}

\newcommand{\HFK}{\mathit{HFK}}

\newcommand{\CFL}{\mathit{CFL}}
\newcommand{\HFL}{\mathit{HFL}}
\newcommand{\HFKh}{\widehat{\HFK}}

\begin{document}

\title{Ribbon Homology Cobordisms and Link Floer Homology}
\author{Gary Guth}

\address{Department of Mathematics, Stanford University}
\email{gmguth@stanford.edu}

\begin{abstract}
We make use of link Floer homology to study cobordisms between links embedded in 4-dimensional ribbon homology cobordisms. Combining results of Daemi--Lidman--Vela-Vick--Wong and Zemke, we show that ribbon homology concordances induce split injections on $\cHFL^-$. We also make use of the torsion submodule of $\HFL^-$ to give restrictions on the number of critical points in ribbon homology concordances.
\end{abstract}

\maketitle

\section{Introduction}\label{Section: Intro}
A concordance $C$ between knots $K_0\sub S^3\times \{0\}$ and $K_1 \sub S^3\times \{1\}$ is an annulus embedded in $S^3\times [0,1]$ such that $K_i = C\cap S^3\times \{i\}$, for $i = 0, 1$. We say that $C$ is a \textit{ribbon concordance} if it has no interior local maxima with respect to the projection $S^3 \times [0, 1] \ra [0, 1]$. Equivalently, a concordance is ribbon if it admits a handle decomposition without any 2-handles. Ribbon concordances are directional. In fact, the notion of ribbon concordance gives rise to preorder on the set of knots in $S^3$: $K_0 \le K_1$ if there is a ribbon concordance \emph{from} $K_0$ \emph{to} $K_1$. Recent work of Agol, confirming a conjecture due to Gordon \cite{Gordon_ribbonconcordances}, states that ribbon concordance defines a partial order \cite{agol_ribbon_concordance}. 

Concordances induce maps on knot Floer homology. Moreover, by the work of Zemke in \cite{zemke_ribbon}, the geometric directionality of ribbon concordances carries through to the algebra, namely, the map induced by a ribbon concordance on knot Floer homology is a split injection. Work of Levine and Zemke proves the analogous result for Khovanov homology \cite{Kh_RibbonConcordance}.

In a similar spirit, one can consider cobordisms between 3-manifolds, which are \emph{homologically} products; a four-dimensional $R$-\textit{homology cobordism} $W$ from $Y_0$ to $Y_1$ is a smooth, compact, oriented $4$-manifold with boundary $-Y_0\cup Y_1$ satisfying $H_*(W, Y_i; R) =0,$ for $i = 0, 1$ and $R$ a ring. Cobordism will often be written as morphisms, $W: Y_0 \ra Y_1$, as cobordisms induce maps between $HF^\circ(Y_0)$ and $HF^\circ(Y_1)$. The directionality of these cobordisms can be taken in to account by studying \textit{ribbon} $R$-\textit{homology cobordisms}, which are $R$-homology cobordisms which admit a handle decomposition without any 3- or 4-handles. In \cite{DLVW_ribbon}, Daemi--Lidman--Vela-Vick--Wong prove that ribbon $\F_2$-homology cobordisms induce split injections on Heegaard Floer homology.

It is natural to ask whether these two results can be combined. Let $L_i \sub Y_i$ be oriented links in $Y_i$ for $i = 0,1$.  A \textit{link cobordism} between $(Y_0, L_0)$ and $(Y_1, L_1)$ is a four dimensional cobordism $W$ from $Y_0$ to $Y_1$ together with an oriented, compact surface $\Sigma$ embedded in $W$ such that $\Sigma \cap Y_i = K_i$ for $i = 0, 1$. We will say that a Morse function $h: W \ra \R$ is \textit{compatible with} $\Sigma$, if the restriction $h|_\Sigma$ is Morse as well.

\begin{defn}\label{def: ribbon link cobordism}
We say $(W, \Sigma): (Y_0, L_0) \ra (Y_1, L_1)$ is a \emph{ribbon link cobordism} if there exists a Morse function $h: W \ra \R$ compatible with $\Sigma$ such that $h$ has no critical points of index $3$ or $4$ and $h|_\Sigma$ has no critical points of index $2$. 
\end{defn}

For a handle theoretic interpretation, see the proof of Theorem \ref{Main Theorem}.

\begin{remark}
In the case $\Sigma$ is an annulus, it is necessary for the Morse function to \emph{simultaneously} give a ribbon structure to $W$ and $\Sigma$. If no restrictions are imposed on the critical points of $h$, one can always find a Morse function on $W$ which restricts to be ribbon on $\Sigma$. See Section \ref{section: applications}.
\end{remark}

\begin{defn}\label{Def: Homology concordance}
If $K_0$ and $K_1$ are knots in 3-manifolds $Y_0$ and $Y_1$ respectively, the pairs $(Y_0, K_0)$ and $(Y_1, K_1)$ are $R$-homology concordant if there is an $R$-homology cobordism $W:Y_0 \ra Y_1$ in which $K_0$ and $K_1$ cobound an embedded annulus $C$. The pair $(W, C)$ is a ribbon $R$-homology concordance if $(W, C)$ is a ribbon link cobordism in the sense of Definition \ref{def: ribbon link cobordism}.
\end{defn}

Situations in which either $W$ or $\Sigma$ are simple with respect to a Morse function will arise frequently enough to justify introducing the following terminology.

\begin{defn}\label{def: morse trivial}
If $(W, \Sigma)$ is a link cobordism equipped with a Morse function $h: W \ra \R$ compatible with $\Sigma$ such that $h$ and $h|_\Sigma$ have no critical points, we say that $(W, \Sigma)$ is \textit{Morse-trivial} with respect to $h$. We say that $(W, \Sigma)$ is \textit{concordance Morse-trivial} with respect to $h$ if $h|_\Sigma$ has no critical points.
\end{defn}

If we say that a pair $(W, \Sigma)$ is Morse-trivial (or concordance Morse-trivial) without reference to a particular Morse function, we simply mean that there exists a Morse function with respect to which $(W, \Sigma)$ has the stated property.

By work of Zemke in \cite{zemke_linkcob}, a link cobordism $(W, \Sigma)$ from $(Y_0, K_0)$ to $(Y_1, K_1)$ induces a map on link Floer homology $\cHFL^-$, once a decoration for $\Sigma$ is chosen (cf. Section \ref{Subsection: Link Floer TQFT}). Our main result states that ribbon $\Z$-homology concordances induce split injections on $\cHFL^-$.

\begin{theorem}\label{Main Theorem}
\label{Main theorem}
Let $(W, \cF): (Y_0, K_0) \ra (Y_1, K_1)$ be a ribbon $\Z$-homology concordance with $\cF = (C, \cA)$ and $\cA$ a pair of parallel arcs. Then the induced map \[F_{W, \cF, \frs}: \cHFL^-(Y_0, K_0, \frs|_{Y_0}) \ra \cHFL^-(Y_1, K_1, \frs|_{Y_1})\] is a split injection.
\end{theorem}

In particular, if knot Floer homology obstructs $K_0$ and $K_1$ from being ribbon concordant in $S^3\times I$ with respect to the usual Morse function, they cannot be ribbon concordant through any ribbon $\Z$-homology cobordism $W: S^3 \ra S^3$. This is perhaps unsurprising, as any such ribbon homology cobordism is necessarily a \emph{homotopy} $S^3\times I$ by \cite[Theorem 1.14]{DLVW_ribbon} . 

A weaker version of our result can be deduced from \cite[Theorem 4.12]{DLVW_ribbon}, which states that if there is a ribbon $\F_2$-homology cobordism from a sutured manifold $(M_0, \eta_0)$ to a sutured manifold $(M_1, \eta_1)$, the sutured Floer homology of the first is isomorphic to a summand of the second. By the isomorphism between $\HFKh$ of a nullhomologous knot $K$ and $SFH$ of its exterior \cite{Juhasz_HolomorphicDisks_SuturedMflds} , it follows that $\HFKh(Y_0, K_0)$ is isomorphic to a summand of $\HFKh(Y_1, K_1)$ whenever there is a cobordism $(W, C)$ from $(Y_0, K_0)$ to $(Y_1, K_1)$ such that exterior of a neighborhood of $C$ is a ribbon $\F_2$-homology cobordism. In particular, this holds whenever there exists a ribbon $\F_2$-homology concordance from $(Y_0, K_0)$ to $(Y_1, K_1)$. \\

Recall that $\HFKh(Y, K)$ of a nullhomologous knot $K\sub Y$ splits as $\bigoplus_j \HFKh(Y, K, j)$, where $j$ is the Alexander grading. By \cite{ni_nonseparatingspheres}, for a nullhomologous knot $K \sub Y$ such that $Y - K$ is irreducible, $g_Y(K) = \max\{j: \HFKh(Y, K, j) \ne 0\}$, where $g_Y(K)$ is the Seifert genus of $K$. An immediate corollary of Theorem \ref{Main theorem} and also \cite[Theorem 4.12]{DLVW_ribbon} is the following.

\begin{corollary} Let $K_0$ and $K_1$ be null-homologous knots 3-manifolds $Y_0$ and $Y_1$ such that $Y_i - K_i$ is irreducible for $i=0,1$. If there is a ribbon $\Z$-homology concordance from $(Y_0, K_0)$ to $(Y_1, K_1)$, then \[g_{Y_0}(K_0) \le g_{Y_1}(K_1),\] where $g_{Y_i}(K_i)$ is the Seifert genus of $K_i$ in $Y_i$. 
\end{corollary}

We will also consider an algebraic reduction of $\cHFL^-(Y, K, \frs)$, denoted $\HFL^-(Y, K, \frs).$  $\HFL^-(Y, K, \frs)$ is a finitely generated $\F[V]$-module, and therefore can be decomposed into a free summand and a torsion summand, which is denoted $\HFL^-_{red}(Y, K, \frs)$. 

\begin{defn}\label{def: torsion order}
Let $K$ be a nullhomologous knot in a 3-manifold $Y$. Define the \textit{torsion order of $K$ in $Y$} to be the quantity
\[
\Ord_V(Y, K, \frs) = \min\{d \in \N: V^d\cdot \HFL^-_{red}(Y, K, \frs) = 0\}.
\]
\end{defn}

Juhász-Miller-Zemke \cite{JMZ_TorsionOrder} use the the torsion order of knots in $S^3$ to give bounds on many topological invariants of knots, including the fusion number, the bridge index, and the cobordism distance. We prove an analogue of \cite[Theorem 1.2]{JMZ_TorsionOrder} in the ribbon homology cobordism setting. 

\begin{theorem}\label{theorem: tors bounds}
Suppose $(W, \Sigma): (Y_0, K_0) \ra (Y_1, K_1)$ is a $\Z$-homology link cobordism such that $W$ is ribbon with respect to a Morse function $h: W \ra \R$ compatible with $\Sigma$. Suppose $\Sigma$ has $m$ critical points of index 0 and $M$ critical points of index 2 with respect to $h|_{\Sigma}$. Then 
\[
\Ord_V(Y_0, K_0, \frs|_{Y_0}) \le \max\{M, \Ord_V(Y_1, K_1)\} + 2g(\Sigma).
\]
When $W$ is a product, we also have
\[
\Ord_V(Y_1, K_1, \frs|_{Y_1}) \le \max\{m, \Ord_V(Y_0, K_0)\} + 2g(\Sigma).
\]
\end{theorem}

We use Theorem \ref{theorem: tors bounds} to prove some results about ribbon cobordisms between knots in homology cobordant 3-manifolds, and consider some generalizations of the fusion number in the context of ribbon homology cobordisms. 

\subsection{Acknowledgements} 
I would like to thank my advisor Robert Lipshitz for his guidance and many suggestions. I would also like to thank Tye Lidman, Maggie Miller, and Ian Zemke for helpful comments. 

\section{Background}\label{Section: Background}

\subsection{The link Floer TQFT}\label{Subsection: Link Floer TQFT}

Knot Floer homology is an invariant of knots in 3-manifolds defined by Ozsváth and Szabó \cite{os_knotinvts} and independently Rasmussen \cite{rasmussen_knotcompl}. The extension to links is due to Ozsváth and Szabó in \cite{os_linkinvts}. We review the definitions in order to establish the conventions we will be following.

\begin{defn}\label{def: mutibased link}
A \textit{multi-based link} $\mathbb{L} = (L, \w, \z)$ in a 3-manifold $Y$ is an oriented link $L$ with two collections of basepoints $\w$ and $\z$ such that each component of $L$ has at least one $\w$- and one $\z$- basepoint and the basepoints alternate between $\w$ and $\z$ as one travels along the link.
\end{defn}

The link Floer groups are constructed by choosing a multi-pointed Heegaard diagram $(\Sigma, \bal, \bbe, \w, \z)$ for $(Y, \L)$, where $\bal = (\al_1, ..., \al_{g+n-1})$ and $\bbe = (\be_1, ..., \be_{g+n-1})$ are the attaching curves, $g$ is the genus of $\Sigma$, and $n = |\w| = |\z|$. Denote by $\T_\al$ and $\T_\be$ the half dimensional tori $\al_1 \times ... \times \al_{g+n-1}$ and $\be_1 \times ... \times \be_{g+n-1}$ in $\Sym^{g+n-1}(\Sigma)$. The link Floer complex splits over $\SpinC$-structures for $Y$ and is generated by intersection points in $\T_\al \cap \T_\be$. In \cite{os_holodisks}, Ozsváth and Szabó define a map 
\[
\frs_\w: \T_\al \cap \T_\be \ra \SpinC(Y)
\]
by interpreting $\SpinC$-structures on $Y$ as homology classes of non-vanishing vector fields on $Y$ (in the sense of \cite{Turaev_Tors_Inv_SpinC}); once we choose a Morse function inducing our Heegaard splitting, an intersection point $\x$ determines flowlines from the index 1 to index 2 critical points and the $\w$ basepoints determine flowlines connecting the index 0 and 3 critical points. Outside a neighborhood of these flowlines, the gradient vector field is non-vanishing, and this homology class is defined to be the $\SpinC$-structure associated to the intersection point $\x$.

Define $\cCFL^-(Y, \frs)$ to be the free $\F_2[U, V]$-module generated by intersection points $\x$ in $\T_\al\cap\T_\be$ with $\frs_\w(\x) = \frs.$ The differential is defined by counting holomorphic disks of Maslov index 1: let

\[
\partial(\x) = \sum_{\y\in \T_\al\cap \T_\be}\sum_{\substack{\phi \in \pi_2(\x, \y),\\ \mu(\phi) =1}} \# \widehat{\mathcal{M}}(\phi)U^{n_\w(\phi)}V^{n_\z(\phi)}\y,
\]
and extend $\F_2[U, V]$-linearly. Note, $\cCFL^-(Y, \frs)$ could also have been defined as generated by intersection points $\x$ with $\frs_\z(\x) = \frs.$ By \cite[Lemma 3.3]{zemke_linkcob}, $\frs_\w(\x) - \frs_\z(\x) = PD[L]$, where $[L]$ is the fundamental class of the link. Hence, when the homology class of the link is trivial in $H_1(Y; \Z)$, the maps $\frs_w$ and $\frs_z$ agree, so either choice yields the same complex. However, for links which are \emph{not} nullhomologous, the two complexes may differ. For a more general set up, see \cite[Section 3]{zemke_linkcob}. \\

Maps between link Floer complexes are induced by decorated link cobordisms.

\begin{defn}\label{Def: decorated link cob}
A \textit{decorated link cobordism} from $(Y_0, \L_0) = (Y_0, (L_0, \w_0, \z_0))$ to $(Y_1, \L_1) = (Y_1, (L_1, \w_1, \z_1))$ is a pair $(W, \cF) = (W, (\Sigma, \cA))$ with the following properties:
\begin{enumerate}
	\item $W$ is an oriented cobordism from $Y_0$ to $Y_1$
	\item $\Sigma$ is an oriented surface in $W$ with $\partial \Sigma = -L_0 \cup L_1$
	\item $\cA$ is a properly embedded 1-manifold in $\Sigma$, dividing it into subsurfaces $\Sigma_\w$ and $\Sigma_\z$ such that $\w_0, \w_1 \sub \Sigma_\w$ and $\z_0,\z_1 \sub \Sigma_\z.$
\end{enumerate}
\end{defn}
In \cite{zemke_linkcob}, it is shown that a decorated link cobordism $(W, \cF)$ from $(Y_0, \L_0)$ to $(Y_1, \L_1)$ and a $\SpinC$-structure $\frs$ on $W$, give rise to a map
\[
F_{W, \cF, \frs}: \cCFL^-(Y_0, \L_0, \frs|_{Y_0}) \ra \cCFL^-(Y_1, \L_1, \frs|_{Y_1}),
\]
and these maps are functorial \cite[Theorem B]{zemke_linkcob} in the following sense: 
\begin{enumerate}
\item Let $(W, \cF)$ be the trivial link cobordism, i.e. $W = Y\times [0, 1]$, $\Sigma = L\times [0, 1]$ and $\cA$ is a collection of arcs $\textbf{p}\times [0, 1]$ where $\textbf{p} \sub L-(\w\cup \z)$ and consists of exactly one point in each component of $L-(\w\cup \z)$. Then $F_{W, \cF, \frs} = \id_{\cCFL^-(Y, \L, \frs|_{Y})}$.
\item If $(W, \cF)$ can be decomposed into the union of two decorated link cobordisms $(W_1, \cF_1) \cup (W_2, \cF_2)$ and $\frs_1$ and $\frs_2$ are $\SpinC$-structures on $W_1$ and $W_2$ respectively which agree on their common boundary, then
\[
F_{W_2, \cF_2, \frs_2}\circ F_{W_1, \cF_1, \frs_1} \simeq \sum_{\substack{\frs \in \SpinC(W),\\ \frs|_{W_i} = \frs_i}}F_{W,\cF, \frs}.
\]

\end{enumerate}

The decorated link cobordism maps are defined as compositions of maps associated to handle attachments to the embedded surfaces and to the ambient 4-manifold. \\

In general, it is quite difficult to compute the decorated link cobordism maps. In some simple cases, however, the link cobordism maps can be computed in terms of the graph cobordism maps defined in \cite{zemke_graphcob}. If $(Y_0, \w_0)$ and $(Y_1, \w_1)$ are 3-manifolds with a collection of basepoints $\w_0$ and $\w_1$, a ribbon graph cobordism between them is a pair $(W, \Gamma)$ such that $W$ is a cobordism from $Y_0$ to $Y_1$ and $\Gamma$ is a graph embedded in $W$ with the properties that $\Gamma \cap Y_i = \w_i$, each basepoint $\w_i$ has valence 1 in $\Gamma$, and at each vertex, the edges of $\Gamma$ are given a cyclic ordering. A ribbon graph cobordism $(W, \Gamma): (Y_0, \w_0) \ra (Y_1, \w_1)$ gives rise to two maps:
\[
F_{W, \Gamma, \frs}^A,F_{W, \Gamma, \frs}^B : CF^-(Y_0, \w_0, \frs|_{Y_0}) \ra CF^-(Y_1, \w_1, \frs|_{Y_1}).
\]
These two maps satisfy
\[
F_{W, \Gamma, \frs}^A \simeq F_{W, \overline{\Gamma}, \frs}^B,
\]
where $\overline{\Gamma}$ is the graph obtained by reversing the cyclic ordering at each of the vertices. The map $F_{W, \Gamma, \frs}^A$ depends on the interaction between the graph and the $\bal$-curves while $F_{W, \overline{\Gamma}, \frs}^B$ depends on the interaction of the graph and the $\bbe$-curves. When $\Gamma$ is simply a path, these maps agree with the original cobordism maps defined by Ozsváth and Szabó \cite[Theorem B]{zemke_graphcob}. 

The graph cobordism maps encode the action of $\Lambda^* H_1(Y)/\tors$ on the Heegaard Floer complexes. Recall that, given a closed loop $\gamma \sub Y$, the action of $[\gamma]\in H_1(Y)/\tors$ on $HF^-(Y, \w)$ is induced by a map
\[
A_\gamma: CF^-(Y, \w) \ra CF^-(Y, \w)
\]
defined bv
\[
A_\gamma(\x) = \sum_{\y\in \T_\al \cap \T_\be} \sum_{\substack{\phi \in \pi_2(\x, \y),\\ \mu(\phi) = 1}} a(\gamma, \phi) \# \widehat{\mathcal{M}}(\phi) U^{n_\w(\phi)}\y.
\]
Roughly speaking, the quantity $a(\gamma, \phi)$ is the intersection number of $\gamma$ and the portion of the boundary of a domain for $\phi$ which lies on an $\bal$-curve. This map satisfies $A_\gamma^2 \simeq 0$ and can be realized by the graph cobordism $(Y \times [0, 1],  \Gamma)$, where the graph $\Gamma$ is shown in Figure \ref{Fig: Graph Action}. 

\begin{figure}
\includegraphics[scale=.12]{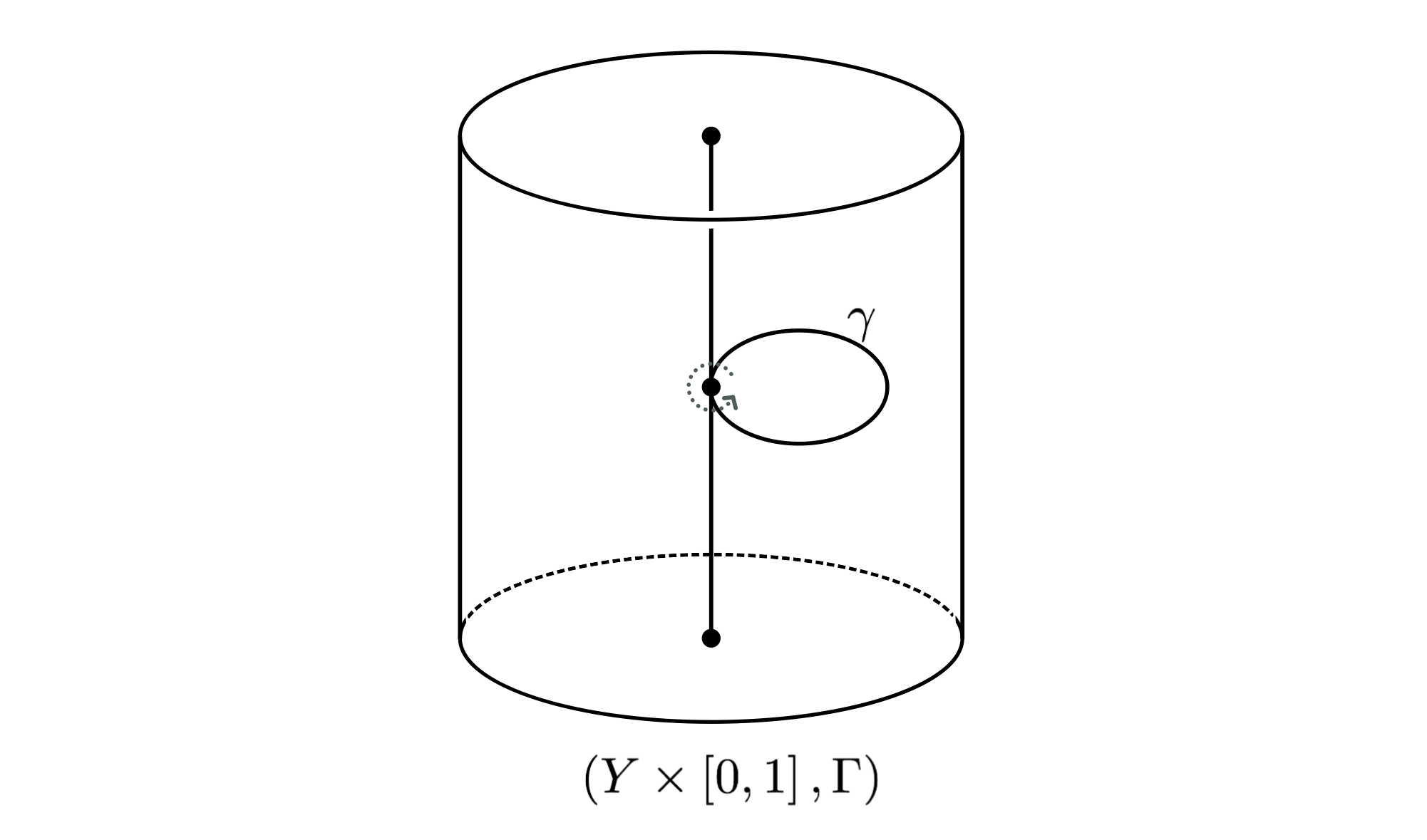}
	\caption{The graph $\Gamma$ realizing the action of a closed curve $\gamma$ in $Y$. The cyclic ordering is indicated by the dashed arrow.}
		\label{Fig: Graph Action}
\end{figure}

If $(W, \cF): (Y_0, \L_0) \ra (Y_1, \L_1)$ is a decorated link cobordism and $\Gamma \sub \Sigma$ is a ribbon graph, we say that $\Gamma$ is the \textit{ribbon 1-skeleton} of $\Sigma_\w$ if $\Gamma \sub W$, $\Gamma \cap Y_i = \w_i$, $\Sigma_\w$ is a regular neighborhood of $\Gamma$ in $\Sigma$, and the cyclic orders of $\Gamma$ agree with the orientation of $\Sigma$. A \textit{ribbon 1-skeleton} of $\Sigma_\z$ is defined in exactly the same way. There are natural chain isomorphisms
\[
\cCFL^-(Y, \L, \frs) \otimes_{\F_2[U, V]} \F_2[U, V]/(V-1) \cong \CF^-(Y, \w, \frt)
\]
and
\[
\cCFL^-(Y, \L, \frs) \otimes_{\F_2[U, V]} \F_2[U, V]/(U-1) \cong \CF^-(Y, \z, \frt-PD[L]).
\]
Under these identifications, a link cobordism map $F_{W, \cF, \frs}$ induces two maps on $\CF^-(Y)$, denoted 
\[
F_{W, \cF, \frs}|_{U=1} \hspace{1cm} \text{ and } \hspace{1cm} F_{W, \cF, \frs}|_{V=1}.
\]
These maps agree with the maps induced by the graph cobordism maps associated to the ribbon 1-skeletons of $\Sigma$.

\begin{theorem}{\cite[Theorem C]{zemke_linkcob}}\label{Thm: Reduction to Graph Cob Maps}
If $(W, \cF)$ is a decorated link cobordism, and $\Gamma_\w \sub \Sigma_\w$ and $\Gamma_\z \sub \Sigma_\z$ are ribbon 1-skeleta, then 
\[
F_{W, \cF, \frs} |_{U=1} \simeq F_{W, \Gamma_\z, \frs - PD[\Sigma]}^A
\]
and
\[
F_{W, \cF, \frs} |_{V=1} \simeq F_{W, \Gamma_\w, \frs}^B,
\]
under the identifications above.
\end{theorem}
For a full discussion on Zemke's graph TQFT framework, see \cite{zemke_graphcob} or, for an overview, see \cite[Section 9.2]{zemke_absgr}.\\

In the following situation, the decorated link cobordism maps are determined by the corresponding graph cobordism map. Let $\cF$ be a closed surface in $W: Y_0 \ra Y_1$, which is decorated by $\cA$, as in Definition \ref{Def: decorated link cob}. Choose disjoint disks $D_0$ and $D_1$ in $\cF$ which each intersect $\cA$ in a single arc, and perturb $\cF$ so that it intersects $Y_i$ in $D_i$. Remove each $D_i$, leaving a decorated cobordism $\cF_0$ between doubly pointed unknots $\U_1$ and $\U_2$. Let $p_i$ denote the center of the disk $D_i$. Identify $\cCFL^-(Y_i, \U_i, \frs)$ with $\CF^-(Y_i, p_i, \frs) \otimes_{\F[W]}  \F[U, V]$, where $W$ acts on $\F[U, V]$ as $UV$. Under this identification, a graph cobordism map $F_{W, \Gamma, \frs}$ induces a map $\cCFL^-(Y_0, \U_0, \frs|_{Y_0}) \ra \cCFL^-(Y_1, \U_1, \frs|_{Y_1})$, which we write as $F_{W, \Gamma, \frs}|^{\F_2[U, V]}.$ In this case, the link cobordism map induced by $(W, \cF_0): (Y_0, \U_0) \ra (Y_1, \U_1)$ is relatively simple. 

\begin{prop}{\cite[Proposition 9.7]{zemke_absgr}}
\label{Prop: Compute closed link cob in terms of graph cob}
Let $\cF = (\Sigma, \cA)$ be a closed decorated link cobordism, and let $(W, \cF_0)$ be the link cobordism obtained from $\cF$ by the procedure outlined above. Define $\Delta A$ to be 
\[
\dfrac{\langle c_1(\frs), \Sigma \rangle - [\Sigma] \cdot [\Sigma]}{2} + \dfrac{\chi(\Sigma_w) - \chi(\Sigma_z)}{2}.
\]
Then, 
\[
F_{W, \cF_0, \frs} \simeq 
	\begin{cases}
		V^{\Delta A} \cdot F_{W, \Gamma_\w, \frs}^B|^{\F[U, V]} & \Delta A \ge 0 \\
		U^{-\Delta A} \cdot F_{W, \Gamma_\z, \frs - PD[\Sigma]}^A|^{\F[U, V]} & \Delta A \le 0,
	\end{cases}
\]
where $\Gamma_w$ and $\Gamma_z$ are ribbon 1-skeleta. 
\end{prop}

It will also be useful to understand how the link cobordism maps change under surgery operations. If $(W, \cF)$ is a link cobordism and $\gamma$ is a closed curve in $\cA$, we can simultaneously do surgery on $\gamma$ in $W$ and $\Sigma$ to obtain a new link cobordism $(W(\gamma), \cF(\gamma))$, i.e. remove a regular neighborhood of $\gamma \sub (W, \Sigma)$, which can can identified with $(S^1\times D^3, S^1 \times D^1)$ and replace it with $(D^2\times S^2, D^2 \times S^0)$. The surface obtained by surgery on $\gamma$ naturally inherits a decoration $\cA(\gamma)$, so denote the new decorated surface $\cF(\gamma) = (\Sigma(\gamma), \cA(\gamma))$ (see Figure \ref{Fig: Surgery Decorations}). If the curve $\gamma$ represents a non-divisible element of $H_1(W; \Z)$ then
\[
F_{W, \cF, \frs} \simeq F_{W(\gamma), \cF(\gamma), \frs(\gamma)}
\]
by \cite[Proposition 5.4]{zemke_connectedsums}. The assumption that $[\gamma]$ is non-divisible guarantees that there is a is the unique $\SpinC$-structure $\frs(\gamma)$ on $W(\gamma)$ which extends a given $\SpinC$-structure on $W-N(\gamma)$.  An analogous result holds for surgeries on collections of curves $\gamma_1, ..., \gamma_n$ which will be of use in the proof of Theorem \ref{Main Theorem}.

\begin{prop}
\label{prop: surgery preserves link cob map}
Let $(W, \cF)$ be a link cobordism. Let $\gamma_1, ..., \gamma_n$ be closed curves in $\cA$ and let $(W(\gamma_1, ..., \gamma_n), \cF(\gamma_1, ..., \gamma_n))$ be the surgered link cobordism. If the restriction map $H^1(W - \amalg N(\gamma_i)) \ra H^1(\amalg \partial N(\gamma_i))$ is surjective, then there is a unique $\SpinC$-structure $\frs(\gamma_1, ..., \gamma_n)$ extending $\frs|_{W - \amalg N(\gamma_i)}$ for each $\frs \in \SpinC(W)$ and
\[
F_{W, \cF, \frs} \simeq F_{W(\gamma_1, ..., \gamma_n), \cF(\gamma_1, ..., \gamma_n), \frs(\gamma_1, ..., \gamma_n)}.
\]
\end{prop}
\begin{proof}
This is effectively proved in \cite[Proposition 5.4]{zemke_connectedsums}, so we only give a sketch. Decompose $(W, \cF)$ as $(W-\amalg N(\gamma_i), \cF-\amalg N(\gamma_i)) \circ (N(\gamma_1), M_1) \circ ... \circ (N(\gamma_n), M_n)$ where $M_i = \cF\cap N(\gamma_i)$ is a tube in $N(\gamma_i)$ with boundary a two component unlink decorated as in Figure \ref{Fig: Surgery Decorations}. Surgery along $\gamma_i$ replaces $(N(\gamma_i), M_i)$ with $(D^2 \times S^2, M'_i)$ where $M'_i$ is a pair of disks each decorated with a single dividing arc (again, see Figure \ref{Fig: Surgery Decorations}). Zemke shows that $(N(\gamma_i), M_i)$ and $(S^2\times D^2, M_i')$ induce the same map on homology. The result follows from the composition law provided the restriction of any $\SpinC$-structure $\frs \in \SpinC(W)$ to $W-\amalg N(\gamma_i)$ extends uniquely to $W(\gamma_1, ..., \gamma_n)$. The restriction of $\frs$ to $\partial N(\gamma_i)$ is torsion, and there is a single $\SpinC$-structure on $D^2 \times S^2$ which is torsion on $\partial N(\gamma_i)$ which we can glue to $\frs|_{W-\amalg N(\gamma_i)}$. The ambiguity to gluing lies in $\delta H^1(\amalg \partial N(\gamma_i))$, which, by considering the Mayer Vietoris sequence, vanishes exactly when the connecting homomorphism is trivial, or equivalently, when the map $H^1(W - \amalg N(\gamma_i)) \ra H^1(\amalg N(\gamma_i))$ is surjective.
\end{proof}

\begin{figure}
\centering
	\includegraphics[scale=.18]{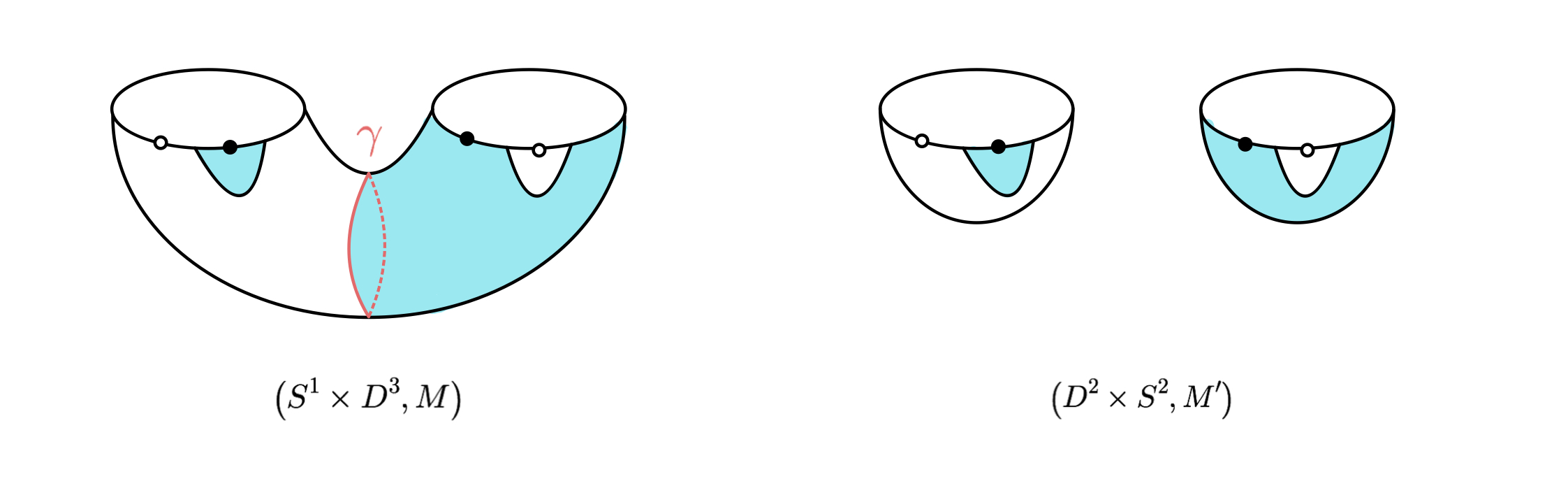}
		\caption{The decorated link cobordisms $(S^1\times D^3, M)$ (left) and $(D^2 \times S^2, M')$ (right). }
			\label{Fig: Surgery Decorations}
\end{figure}

A version of the ``sphere tubing" property of the link cobordism maps \cite[Lemma 3.1]{zemke_ribbon}, \cite[Lemma 4.2]{millerzemke_stronglyhomotopy} will be useful; if $(W, \cF)$ is a link cobordism in a homology cobordism and $S$ is a null-homologous 2-sphere embedded in the complement of the link cobordism, $S$ can be tubed to the embedded surface without changing the induced map.

\begin{prop}
\label{prop: sphere tubing}
Let $\cF = (\Sigma, \cA)$ be a decorated link cobordism in a homology cobordism $W$, and let $S \sub W$ be a smoothly embedded, nullhomologous sphere disjoint from $\cF$. Let $\cF'$ be a decorated cobordism obtained by connecting $\Sigma$ and $S$ by a tube whose feet are disjoint from $\cA$. Then,  \[ F_{W, \cF, \frs} \simeq F_{W, \cF', \frs}.\]
\end{prop}

\begin{proof}
Factor $F_{W, \cF, \frs}$ and $F_{W, \cF', \frs}$ through a regular neighborhood $N(S)$ of $S$. Since $S$ is nullhomologous, $N(S)$ can be identified with $D^2\times S^2$. $\cF'$ intersects $N(S)$ in a disk $D'$ and $\partial N(S)$ in an unknot. We can perturb $\cF$ so it meets $N(S)$ in a disk $D$ and $\partial N(S)$ in an unknot as well. Let $\cD$ and $\cD'$ be the disks $D$ and $D'$ decorated with a single dividing arc. Since $S$ is nullhomologous, the restriction of a given $\frs \in \SpinC(W)$ to $N(S)$ will be torsion. By \cite[Lemma 3.1]{zemke_ribbon}, $F_{N(S), \cD, \frs|_{N(S)}}$ does not depend on the choice of embedded disk, and so $F_{N(S), \cD, \frs|_{N(S))}} \simeq F_{N(S), \cD', \frs|_{N(S)}}$.

Moreover, since $W$ is a homology cobordism, the map $H^2(W) \ra H^2(\partial_-W)$ is an isomorphism, and, in particular, an injection. It follows from the following diagram that the map $H^2(W) \ra H^2(W - N(S))$ is injective as well.

\begin{center}
\begin{tikzcd}
H^2(W) \ar[r, "\cong"] \ar[d] &
	H^2(\partial_- W) \\

H^2(W-N(S)) \ar[ur]
\end{tikzcd}
\end{center}

Therefore, a given $\SpinC$-structure on $W - N(S)$ will extend over $N(S)$, and moreover will extend uniquely. By the composition law for the link cobordism maps, 

\begin{align*}
F_{W, \cF, \frs} &\simeq F_{W-N(S), \cF-N(S), \frs|_{W-N(S)}}\circ F_{N(S), \cD, \frs|_{N(S))}}\\
& \simeq F_{W-N(S), \cF-N(S), \frs|_{W-N(S)}}\circ F_{N(S), \cD', \frs|_{N(S))}} \\
&\simeq F_{W, \cF', \frs},
\end{align*}
as desired.
\end{proof}

Finally, recall in \cite[Proof of Theorem 3.1]{os_holotri}, Ozsváth and Szabó define an extended cobordism map:
\[
F_{W, \frs}: \Lambda^*H_1(W;\Z)/\tors\otimes\CF^-(Y_0, \frs|_{Y_0})\ra \CF^-(Y_1, \frs|_{Y_1}).
\]
A Heegaard triple $(\Sigma, \bal, \bbe, \bga)$ gives rise to a cobordism $X_{\alpha, \beta, \gamma}$. Since the natural map $H_1(\partial X_{\alpha, \beta, \gamma})\ra H_1(X_{\alpha, \beta, \gamma})$ is surjective, a given element $h \in H_1(X_{\alpha, \beta, \gamma})$ is in the image of some $(h_1, h_2, h_3) \in H_1(\partial X_{\alpha, \beta, \gamma})\cong H_1(Y_{\alpha, \beta}) \oplus H_1(Y_{\beta, \gamma})\oplus H_1(Y_{\alpha, \gamma})$. Then, by utilizing the $H_1/\tors$-action on $\partial X_{\alpha, \beta, \gamma}$, define a map 
\[
 \Lambda^*H_1(X_{\alpha, \beta, \gamma};\Z)/\tors\otimes\CF^-(Y_{\alpha, \beta}, \frs_{\alpha, \beta})\otimes\CF^-(Y_{\beta, \gamma}, \frs_{\beta, \gamma}) \ra \CF^-(Y_{\alpha, \gamma}, \frs_{\alpha, \gamma}), 
\]
by 
\[F_{\alpha, \beta, \gamma}(h\otimes \x\otimes \y) = F_{\alpha, \beta, \gamma}((h_1\cdot \x)\otimes \y ) + F_{\alpha, \beta, \gamma}(\x\otimes (h_2\cdot \y)) - h_3\cdot F_{\alpha, \beta, \gamma}(\x\otimes \y).\] 
This action induces a map on homology. By decomposing $W = W_1 \cup W_2 \cup W_3$ into the 1-, 2-, and 3-handle attachment cobordisms, the extended cobordism map is defined to be
\[
F_{W, \frs}(h \otimes \x) = F_{W_3, \frs}\circ F_{W_2, \frs}(h\otimes F_{W_1, \frs}(\x) ),
\]
This map satisfies a version of the usual $\SpinC$-composition law: 
\begin{prop}{\cite[Proposition 4.20]{os_holotri}}
If $W = W_1 \cup W_2$, and $\xi_1 \in \Lambda^*H_1(W_1; \Z)/\tors$ and $\xi_2 \in \Lambda^*H_1(W_2; \Z)/\tors$, then
\[
F_{W_2, \frs_2}(\xi_2 \otimes F_{W_1, \frs_1}(\xi_1 \otimes \cdot)) = \sum_{\substack{\frs \in \SpinC(W),\\ \frs|_{W_i} = \frs_i}} F_{W, \frs}((\xi_3 \otimes \cdot),
\]
where $\xi_3\in \Lambda^*H_1(W; \Z)/\tors$ is the image of $\xi_1 \otimes \xi_2$ under the natural map.
\end{prop}

There is an $H_1(Y; \Z)/\tors$-action on multi-pointed Heegaard diagrams as well \cite[Equation 5.2]{zemke_graphcob}
\[
A_\gamma: \cCFL^-(Y, \L, \frs) \ra \cCFL^-(Y, \L, \frs),
\]
defined 
\[
A_\gamma(\x) = \sum_{\y\in \T_\al \cap \T_\be} \sum_{\substack{\phi \in \pi_2(\x, \y),\\ \mu(\phi) = 1}} a(\gamma, \phi) \# \widehat{\mathcal{M}}(\phi) U^{n_\w(\phi)}V^{n_\z(\phi)}\y.
\]
Using this action, we can define extended link cobordism maps 
\[
F_{W, \cF, \frs}: \left(\Lambda^* H_1(W)/\tors \otimes \F_2 \right) \otimes \cCFL^-(Y_0, \L_0, \frs|_{Y_0}) \ra \cCFL^-(Y_1, \L_1, \frs|_{Y_1}).
\]
in exactly the same way. Note, that since the link Floer TQFT is defined with coefficients in $\F_2$, we need to tensor the exterior algebra generated by  $H_1(Y)/\tors$ with $\F_2$.\\

\section{First Homology Action as Link Cobordism Maps}\label{Section: Homology Action}

It is helpful in geometric arguments that the $H_1/\tors$-action can be realized as a graph cobordism map on $\CF^-$. In the same way, it is beneficial to realize the $H_1(Y)/\tors$-action on $\cCFL^-$ as a link cobordism map. By Theorem \ref{Thm: Reduction to Graph Cob Maps}, $\Sigma_\w$ and $\Sigma_\z$ should be ribbon 1-skeleta of the graph in Figure \ref{Fig: Graph Action}. Our strategy will be to try to compute the decorated link cobordism map obtained by tubing on a torus whose longitude is a curve which represents the class in $H_1(Y)/\tors$. 

\begin{figure}
\centering
     \includegraphics[scale = .15]{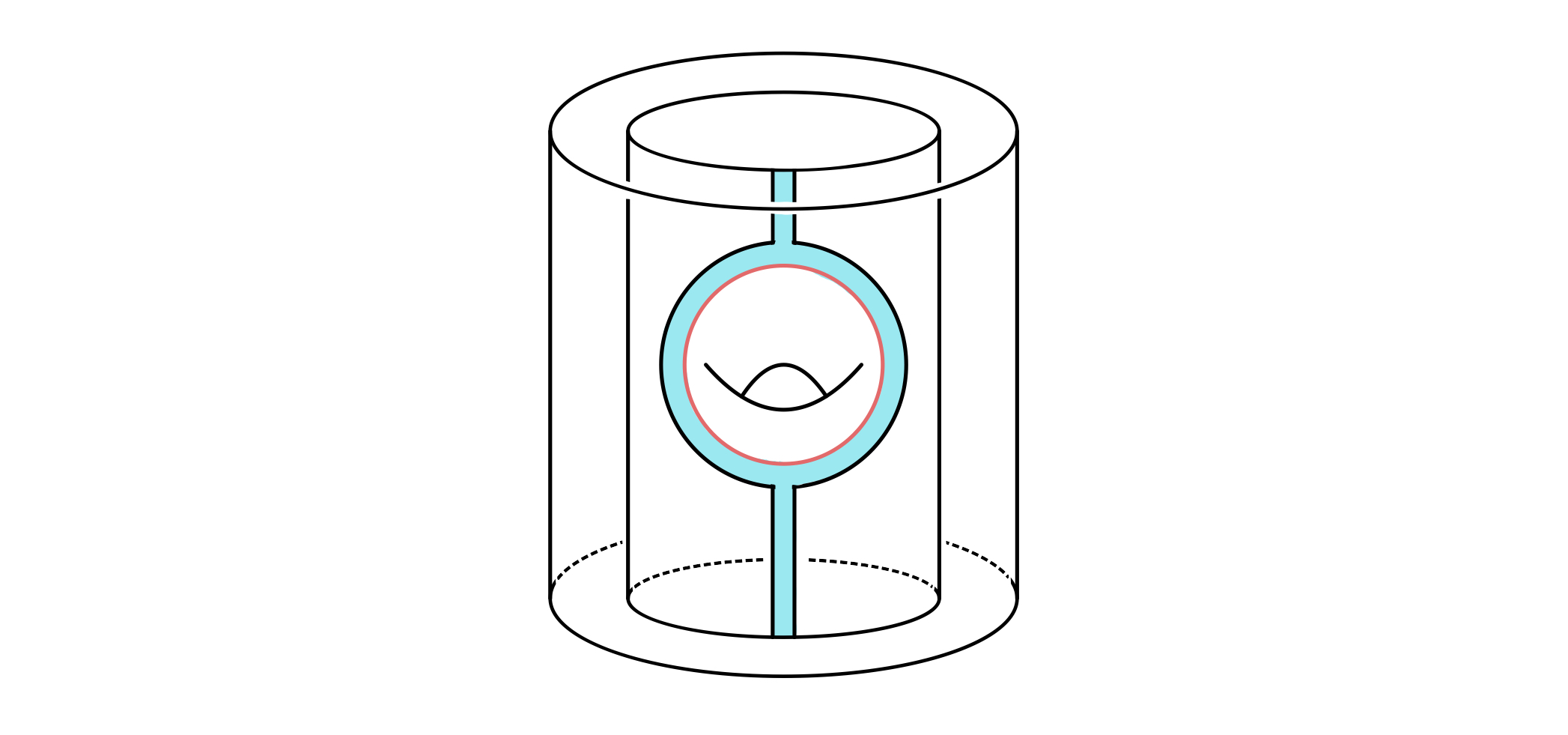}
          \caption{A schematic of the decorated surface in $S^2\times S^1\times \{0\}$ described in the text preceding Lemma \ref{Lemma: H1 Action on S1x S2}.}
               \label{Fig: Decorated surface in S1xS2}
\end{figure}

Construct a closed decorated link cobordism $\cF$ inside of $S^2\times S^1\times [-1,1]$ as follows. Let $\Sigma$ be the boundary of a tubular neighborhood of $\pt \times S^1\times \{0\}\sub S^2\times S^1\times \{0\}$. Let the dividing circles $\cA$ be two parallel closed curves on the boundary of $\Sigma$ obtained by isotoping $\pt \times S^1 \times \{0\}$ radially, so that $\cA$ divides $\Sigma$ into $\Sigma_\w$ and $\Sigma_\z$ which are both annuli. As in the text preceding Proposition \ref{Prop: Compute closed link cob in terms of graph cob}, choose disks $D_1$ and $D_2$ which intersect the same dividing arc $\pt \times S^1$, and take $(W, \cF_0)$ to be the link cobordism obtained by isotoping $\cF$ and removing the disks $D_1$ and $D_2$. See Figure \ref{Fig: Decorated surface in S1xS2}. Note that the graph shown in Figure \ref{Fig: Graph Action} is a ribbon 1-skeleton for both $\Sigma_\w$ and $\Sigma_\z$.

\begin{lemma}
\label{Lemma: H1 Action on S1x S2}
For $(S^2\times S^1\times [-1,1], \cF_0)$ the link cobordism described above 
\[F_{S^2\times S^1 \times [-1,1], \cF_0, \frs_0}(\theta^+) = \theta^- \text{      and     } F_{S^2\times S^1 \times [-1,1], \cF_0, \frs_0}(\theta^-) = 0,\]
where $\theta^+$ ($\theta^-$) is the generator of $\cCFL^-(S^2 \times S^1, \U, \frt_0)$ of higher (lower) grading, $\frs_0$ is the torsion $\SpinC$-structure on $S^2 \times S^1 \times [-1,1]$ and $\frt_0 = \frs_0|_{S^2 \times S^1}$.
\end{lemma}

\begin{proof}
By Proposition \ref{Prop: Compute closed link cob in terms of graph cob}, $F_{S^2\times S^1 \times [-1,1], \cF_0, \frs_0}$ is determined by its reduction to a graph cobordism map: concretely, if $F_{S^2\times S^1 \times [-1,1],\Gamma_w, \frs}^B$ is the corresponding graph map, then
\[F_{S^2\times S^1 \times [-1,1], \cF_0, \frs_0} \simeq 
V^{\Delta A} \cdot F^B_{S^2\times S^1 \times [-1,1], \Gamma_w, \frs_0}|^{\F_2[U, V]} {}
\]
under the identification of $\cCFL^-(Y, \mathbb{U}, \frs_0)$ with $\CF^-(Y, w, \frs_0) \otimes_{\F_2[W]} \F_2[U, V]$  as before. The quantity $\Delta A$, given by 
\[
\dfrac{\langle c_1(\frs_0), [\Sigma] \rangle - [\Sigma] \cdot [\Sigma]}{2} + \dfrac{\chi(\Sigma_w) - \chi(\Sigma_z)}{2},
\]
vanishes, since $[\Sigma]$ is nullhomologous in $S^1 \times S^2$, and $\Sigma_w$ and $\Sigma_z$ are both cylinders. By construction, $F^B_{S^2\times S^1 \times [-1,1], \Gamma_w, \frs_0}$ is the graph cobordism shown in Figure \ref{Fig: Graph Action}. Hence, $F^B_{S^2\times S^1 \times [-1,1], \Gamma_w, \frs_0}$ is just the map $A_\gamma$. It is straightforward to verify that the action of $[\pt \times S^1] \in H_1(S^2\times S^1)$ takes $\theta^+$ to $\theta^-$ and $\theta^-$ to zero.
\end{proof}

We now turn to the case of a nullhomologous knot $K$ embedded in an arbitrary 3-manifold $Y.$ Let $(Y \times [-1, 1], \cF_Y)$ be Morse-trivial, in the sense defined in Section \ref{Section: Intro}. The idea is to modify $\cF_Y$ by by tubing on a torus with a dividing arc which represents a class in $H_1(Y;\Z)/\tors$. Let $\gamma \sub Y\times \{0\} - (K\times \{0\})$ be such a curve. Choose a path $\lambda$ in $Y \times \{0\}$ connecting $\gamma$ to a point $p$ in $K \times \{0\}$ which lies on one of the dividing arcs. 

Let $T$ be the boundary of a tubular neighborhood of $\gamma$ in $Y \times \{0\}$. Decorate $T$ with two parallel circles isotopic to $\gamma$. Tube $T$ and $\Sigma$ together along $\lambda$. Denote the resulting surface $\Sigma_\gamma$. Decorate the tube with two parallel arcs, connected to one of the circles parallel to $\gamma$ on one end, and to one of the dividing arcs on $\cF_Y$ on the other end. A schematic of this decoration, which is denoted $\cA_\gamma$, is shown in Figure \ref{Fig: Schematic of F_Gamma}.

\begin{figure}
	\centering
		\includegraphics[scale=.13]{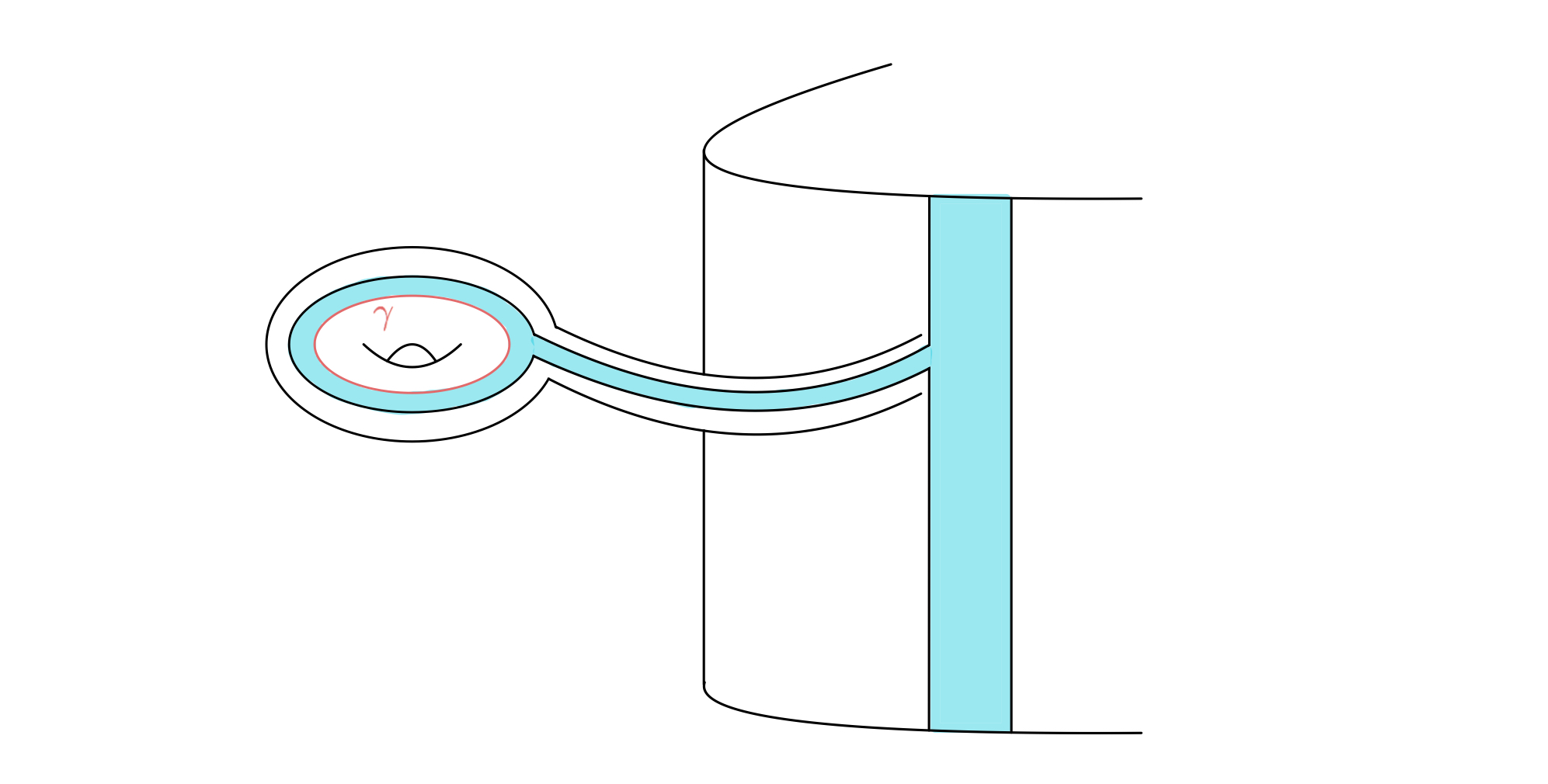}
			\caption{A schematic of the decorated surface $\cF_\gamma$. In general, $K\times \{0\}$ and $\gamma$ might be linked.}
				\label{Fig: Schematic of F_Gamma}
\end{figure}

\begin{lemma}
\label{lemma: product action}
For the decorated surfaces $(Y\times [-1,1], \cF_Y)$ and $(Y \times [-1,1], \cF_\gamma)$ described above, we have that
\[
F_{Y\times [-1,1], \cF_\gamma, \frs}(\cdot) \simeq F_{Y\times [-1,1], \cF_Y, \frs}([\gamma]\otimes \cdot).\]
\end{lemma}

\begin{proof}
Decompose $(Y\times [-1,1], \cF_\gamma)$ as $(X, \cF_X) \circ (N_\lambda(\gamma), \cF)$ where $X= ((Y \times [-1,1]) - N_\lambda(\gamma))$, $\cF = \cF_\gamma \cap N_\lambda(\gamma)$, and $\cF_X = \cF_\gamma \cap X$. $\cF$ is a punctured torus whose decoration is shown in Figure \ref{Fig: Comparison of Tubed Link Cob with Graph Action}.  Given a $\SpinC$-structure $\frs$ on $Y \times [-1,1]$, its restriction to $X$ is torsion on $\partial N_\lambda(\gamma)$, which can be extended by $\frs_0$, the unique $\SpinC$-structure on $N_\lambda(\gamma)$. By considering another Mayer-Vietoris sequence, it is not hard to see that this extension is unique. The composition law then implies that \[F_{Y\times [-1,1], \cF_\gamma, \frs} \simeq F_{X, \cF_X, \frs|_X} \circ F_{N_\lambda(\gamma), \cF, \frs_0}.\]

The map $F_{N_\lambda(\gamma), \cF, \frs_0}$ associated to the cobordism $(Y, K) \ra (Y \amalg S^2 \times S^1, K \amalg U)$ can be computed as the composition of a 0-handle/birth map, followed by a 1-handle map, followed by the map $F_{S^2\times S^1 \times [-1,1], \cF_0, \frs_0}$ computed above: given an element $\x \in \cCFL^-(Y, \L, \frs)$, the 0-handle/birth map simply introduces a pair of intersection points $c^+, c^-$ on a genus 0 Heegaard diagram for $S^3$, and takes $\x \mapsto \x \otimes c^+$. Attaching a 1-handle, with both feet attached to the new 0-handle corresponds to the map $\x\otimes c^+ \mapsto \x\otimes \theta^+$ \cite[Section 5]{zemke_linkcob}. By Lemma \ref{Lemma: H1 Action on S1x S2}, $F_{S^2\times S^1 \times [-1,1], \cF_0, \frs_0}(\x \otimes \theta^+) = \x \otimes \theta^-$. All together then, \[F_{N_\lambda(\gamma), \cF, \frs_0}(\x) = \x \otimes \theta^-.\] 
On the other hand, consider the cobordism $F_{N_\lambda(\gamma), \cD, \frs_0}$ where $\cD$ is a disk decorated with a single arc, followed by the action of $[\gamma]$: this can be computed as the composition of the 0-handle/birth and 1-handle maps followed by the action of $[\gamma]$. Just as before, the 0-handle/birth and 1-handle maps take an intersection point $\x$ to $\x \otimes \theta^+$ and the action of $[\gamma]$ takes this element to $\x \times \theta^-$. Hence, 
\[
F_{N_\lambda(\gamma), \cF, \frs_0}(\x) = [\gamma] \cdot F_{N_\lambda(\gamma), \cD, \frs_0}(\x).
\]
See Figure \ref{Fig: Comparison of Tubed Link Cob with Graph Action} for a comparison of these two cobordism maps. 

By definition of the extended link cobordism maps, $[\gamma]\cdot F_{N_\lambda(\gamma), \cD, \frs_0}(\x) = F_{N_\lambda(\gamma), \cD, \frs_0}( [\gamma] \otimes \x)$. Combining these observations with the composition law for the extended cobordism maps shows
\begin{align*}
F_{Y\times [-1,1], \cF_\gamma, \frs}(\x) &\simeq F_{X, \cF_X, \frs|_X} \circ F_{N_\lambda(\gamma), \cF, \frs_0}(\x) \\
&\simeq F_{X, \cF_X, \frs|_X} \circ F_{N_\lambda(\gamma), \cD, \frs_0}([\gamma] \otimes \x) \\
&\simeq F_{X, \cF_X, \frs_X}(1\otimes F_{N_\lambda(\gamma), \cD, \frs_0}([\gamma] \otimes \x))\\
& \simeq F_{Y\times [-1,1], \cF_Y, \frs}([\gamma] \otimes \x)
\end{align*} 
as desired.
\end{proof}

\begin{figure}
     \centering
          \includegraphics[scale=.18]{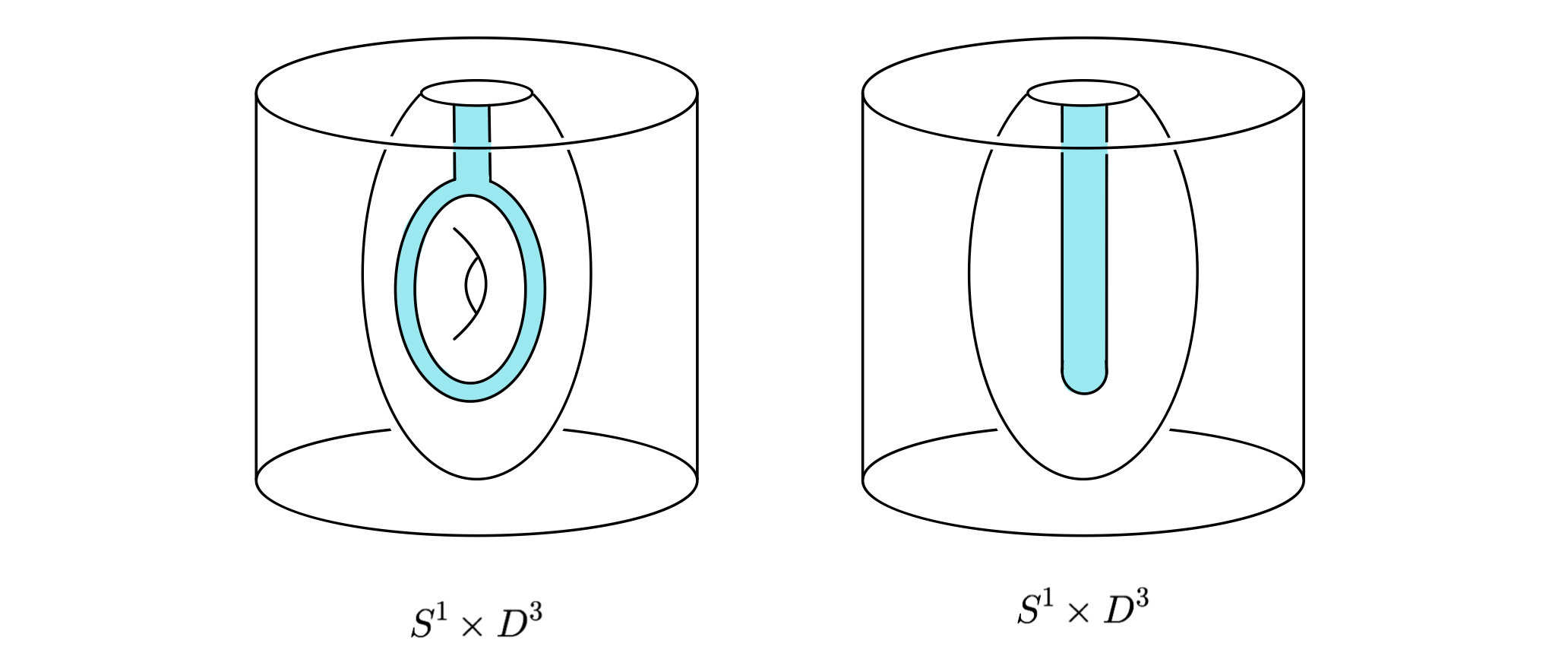}
               \caption{When the link cobordism map on the right is followed by the action of $\gamma$, it becomes equivalent to the map on the left.}
                    \label{Fig: Comparison of Tubed Link Cob with Graph Action}
\end{figure}

Note that the choice of path $\lambda$ did not matter, since the diffeomorphism type of the neighborhood $N_\lambda(\gamma)$ did not depend on $\lambda$. 

Let $(W, \cF): (Y_0, K_0) \ra (Y_1, K_1)$ be a link cobordism which is concordance Morse-trivial. Decompose $W$ as a composition of handle attachments $W_3 \circ W_2 \circ W_1$.  Let $\gamma$ be a curve in $W$ which represents an element of $H_1(W)/\tors$. Homotope the curve $\gamma$ so it is contained in the boundary of $W_1$, which is denoted $\widetilde{Y} = \partial_+W_1$. In this dimension and codimension, such a homotopy can be taken to be an isotopy. Let $\widetilde{\cF}_\gamma$ be the decorated surface in $\widetilde{Y}\times [-1,1]$ described above which realizes the action of $[\gamma]$. Let $(W, \cF_\gamma)$ be the link cobordism $(W_3, \cF_3)  \circ (W_2, \cF_2)  \circ (\widetilde{Y}\times [-1,1], \widetilde{\cF}_\gamma)  \circ (W_1, \cF_1)$, where $\cF_i$ is $\cF\cap W_i.$

\begin{lemma}
\label{lemma: homology action}
Let $(W, \cF_\gamma)$ be the decorated surface described above. Then, \[F_{W, \cF, \frs}([\gamma] \otimes \cdot) \simeq F_{W, \cF_\gamma, \frs}(\cdot).\]
\end{lemma} 

\begin{proof}
We will make use of the decomposition of $(W, \cF_\gamma)$ as
\[
(W_3, \cF_3)  \circ (W_2, \cF_2)  \circ (\widetilde{Y}\times [-1,1], \widetilde{\cF}_\gamma)  \circ (W_1, \cF_1).
\]
$F_{W, \cF_\gamma}$ can be computed as the composition:
\[
F_{W, \cF_\gamma, \frs}(x) \simeq F_{W_3, \cF_3, \frs|_{W_3}}  \circ F_{W_2, \cF_2, \frs|_{W_2}}  \circ F_{\tilde{Y}\times [-1,1], \cF_\gamma, \frs|_{\tilde{Y}\times [-1,1]}}  \circ F_{W_1, \cF_1, \frs|_{W_1}}(x).
\]
By Lemma \ref{lemma: product action}, this map is chain homotopic to 
\[
F_{W_3, \cF_3, \frs|_{W_3}}  \circ F_{W_2, \cF_2, \frs|_{W_2}}  \circ F_{\tilde{Y}\times [-1,1], \cF_{\tilde{Y}}, \frs|_{\tilde{Y}\times [-1,1]}} ([\gamma] \otimes F_{W_1, \cF_1, \frs|_{W_1}}(x)).
\]
Since $(\tilde{Y} \times [-1,1], \cF_{\tilde{Y}})$ is the identity, \[F_{\tilde{Y}\times [-1,1], \cF_{\tilde{Y}}, \frs|_{\tilde{Y}\times [-1,1]}} ([\gamma] \otimes F_{W_1, \cF_1, \frs|_{W_1}}(x)) = [\gamma] \cdot F_{W_1, \cF_1, \frs|_{W_1}}(x) + 0,\] and hence this map can be rewritten as:
\[
F_{W_3, \cF_3, \frs|_{W_3}}  \circ F_{W_2, \cF_2, \frs|_{W_2}} ([\gamma] \cdot  F_{W_1, \cF_1, \frs|_{W_1}}(x)).
\]
But, since $\gamma$ was chosen as to lie in $\widetilde{Y}$, this is by definition equal to the map $F_{W, \cF, \frs}([\gamma] \otimes x)$.
\end{proof}

\section{Proof of Main Theorem}\label{section: proof of main thm}

As is typical in proving results of this kind, our strategy will be to compare the double of a ribbon $\Z$-homology concordance to a Morse-trivial link cobordism. Recall that the double of a link cobordism $(W, \Sigma): (Y_0, K_0) \ra (Y_1, K_1)$ is the link cobordism $(D(W), D(\Sigma)) = (W, \Sigma) \cup_{(Y_1, K_1)} (\overline{W}, \overline{\Sigma})$, where $(\overline{W}, \overline{\Sigma})$ is the link cobordism obtained by turning $(W, \Sigma)$ around and reversing the orientation. 

It will be helpful to have a version of the ``sphere tubing" property of the link cobordism maps \cite[Lemma 3.1]{zemke_ribbon}, \cite[Lemma 4.2]{millerzemke_stronglyhomotopy}; if $(W, \cF)$ is a link cobordism in a homology cobordism and $S$ is a null-homologous 2-sphere embedded in the complement of the link cobordism, $S$ can be tubed to the embedded surface without changing the induced map.

\begin{prop}
\label{prop: sphere tubing}
Let $\cF = (\Sigma, \cA)$ be a decorated link cobordism in a homology cobordism $W$, and let $S \sub W$ be a smoothly embedded, nullhomologous sphere disjoint from $\cF$. Let $\cF'$ be a decorated cobordism obtained by connecting $\Sigma$ and $S$ by a tube whose feet are disjoint from $\cA$. Then,  \[ F_{W, \cF, \frs} \simeq F_{W, \cF', \frs}.\]
\end{prop}

\begin{proof}
Factor $F_{W, \cF, \frs}$ and $F_{W, \cF', \frs}$ through a regular neighborhood $N(S)$ of $S$. Since $S$ is nullhomologous, $N(S)$ can be identified with $D^2\times S^2$. $\cF'$ intersects $N(S)$ in a disk $D'$ and $\partial N(S)$ in an unknot. We can perturb $\cF$ so it meets $N(S)$ in a disk $D$ and $\partial N(S)$ in an unknot as well. Let $\cD$ and $\cD'$ be the disks $D$ and $D'$ decorated with a single dividing arc. Since $S$ is nullhomologous, the restriction of a given $\frs \in \SpinC(W)$ to $N(S)$ will be torsion. By \cite[Lemma 3.1]{zemke_ribbon}, $F_{N(S), \cD, \frs|_{N(S)}}$ does not depend on the choice of embedded disk, and so $F_{N(S), \cD, \frs|_{N(S))}} \simeq F_{N(S), \cD', \frs|_{N(S)}}$.

Moreover, since $W$ is a homology cobordism, the map $H^2(W) \ra H^2(\partial_-W)$ is an isomorphism, and, in particular, an injection. It follows from the following diagram that the map $H^2(W) \ra H^2(W - N(S))$ is injective as well.

\begin{center}
\begin{tikzcd}
H^2(W) \ar[r, "\cong"] \ar[d] &
	H^2(\partial_- W) \\

H^2(W-N(S)) \ar[ur]
\end{tikzcd}
\end{center}

Therefore, a given $\SpinC$-structure on $W - N(S)$ will extend over $N(S)$, and moreover will extend uniquely. By the composition law for the link cobordism maps, 

\begin{align*}
F_{W, \cF, \frs} &\simeq F_{W-N(S), \cF-N(S), \frs|_{W-N(S)}}\circ F_{N(S), \cD, \frs|_{N(S))}}\\
& \simeq F_{W-N(S), \cF-N(S), \frs|_{W-N(S)}}\circ F_{N(S), \cD', \frs|_{N(S))}} \\
&\simeq F_{W, \cF', \frs},
\end{align*}
as desired.
\end{proof}

We will prove Theorem \ref{Main theorem} in two steps: we will prove the theorem for link cobordisms $(W, \Sigma)$ which are concordance Morse-trivial and then argue that we can always reduce to this case. 

The first step will follow from Proposition \ref{prop: surgery preserves link cob map}. Let $(W, \cF): (Y_0, K_0) \ra (Y_1, K_1)$ be a ribbon $\Z$-homology concordance which is concordance Morse-trivial and decorated by a pair of parallel arcs. Decompose $W = W_1\cup W_2$ into the 1- and 2-handle cobordisms. A key observation of \cite{DLVW_ribbon} is that the product cobordism $Y_0 \times [-1,1]$ and the double of $W$ can be obtained by two different surgeries on the same intermediate manifold $X = D(W_1)$. $X$ can be described explicitly as $(Y_0 \times [-1,1]) \#^n (S^1 \times S^3)$. It is not hard to see that surgery on $S^1 \times S^3$ along $S^1 \times \pt$ yields $S^4$, and so that $Y_0\times [-1,1]$ can be obtained from surgery on $X$ is straightforward. That $D(W)$ can also be obtained by surgery on $X$ follows from the following lemma.

\begin{lemma}\label{lemma: Surgery on doubles}\cite[Proposition 5.1]{DLVW_ribbon}
Let $W: Y_0 \ra Y_1$ be a cobordism corresponding to attaching 2-handles along curves $\gamma_1, \dots, \gamma_n \sub Y_0$. Then, the double of $W$ can be obtained from $Y_0 \times [-1,1]$ by doing surgery on $\gamma_1, \dots, \gamma_n\sub Y_0 \times \{0\}$.
\end{lemma}

To apply Proposition \ref{prop: surgery preserves link cob map}, there is a homological condition on the collection of surgery curves $\alpha_1, \dots, \alpha_n$, which must be satisfied, namely the restriction map $H^1(X - \amalg N(\alpha_i)) \ra H^1(\amalg \partial N(\alpha_i))$ must be surjective. 

\begin{lemma}\label{lemma: restriction maps surjective}
If $\alpha_1, \dots, \alpha_n\sub X$ are either the attaching curves of the 2-handles of $W$ or the core curves $S^1\times \pt$ of the $S^1\times S^3$ summands, which we denote $\gamma_1, ..., \gamma_n$ and $\eta_1, ..., \eta_n$ respectively, the restriction map $H^1(X - \amalg N(\alpha_i)) \ra H^1(\amalg \partial N(\alpha_i))$ is surjective.
\end{lemma}
\begin{proof}
Let $\alpha_1, \dots, \alpha_n$ be either set of curves. Inclusions induce the following commutative diagram:

\begin{center}
	\begin{tikzcd}
	H^1(X) \ar[r] \ar[d] &
		H^1(\amalg \partial N(\alpha_i) ) \\
	H^1(W - \amalg N(\alpha_i)) \ar[ur]
	\end{tikzcd}
\end{center}

The map $H^1(X - \amalg N(\alpha_i)) \ra H^1(\amalg \partial N(\alpha_i))$ is surjective if the map $H^1(X) \ra H^1(\amalg \partial N(\alpha_i))$ is. The curve $\eta_i$ runs over the $i$th 1-handle geometrically once, and since $W$ is a $\Z$-homology cobordism, after some handle slides, we can arrange that the curve $\gamma_i$ runs over the $i$th 1-handle algebraically once. In either case, this implies that the composition
\[
H^1(\#^n(S^1\times S^3)) \ra H^1(X) \ra H^1(\amalg \partial N(\alpha_i)),
\]
is an isomorphism. Therefore, the map $H^1(X) \ra H^1(\amalg \partial N(\alpha_i))$ is surjective as desired.\\
\end{proof}

We can now establish the theorem for the case where $(W, \Sigma)$ is concordance Morse-trivial.

\begin{prop}\label{prop: trivial concordance in W}
Suppose $(W, \cF): (Y_0, K_0) \ra (Y_1, K_1)$ is a ribbon $\Z$-homology concordance which is concordance Morse-trivial where $\cF = (C, \cA)$ is the concordance decorated by a pair of parallel arcs. Then, every $\frs\in \SpinC(W)$ has a unique extension $D(\frs)\in \SpinC(D(W))$ and the map induced by the double of $(W, \cF)$ \[F_{D(W), D(\cF), D(\frs)}: \cHFL^-(Y_0, K_0, \frs|_{Y_0}) \ra \cHFL^-(Y_0, K_0, \frs|_{Y_0})\] is the identity.
\end{prop}

\begin{proof}
Decompose $W$ as $W_1 \cup W_2$ where $W_i$ is the cobordism corresponding to the attachment of the $i$-handles, $i = 1, 2$. Let $\widetilde{Y} = \partial_+ W_1$ which can be identified with $Y_0 \#^n (S^1 \times S^2)$, where $n$ is the number of $1$-handles (and since $W$ is a $\Z$-homology cobordism, $n$ is also the number of 2-handles). Let $X = \overline{W_1}\cup W_1$ be the double of $W_1$, which is diffeomorphic to $(Y_0 \times [-1,1])\#^n (S^1 \times S^3)$. Define a decorated surface $\cF_\alpha = (\Sigma_\alpha, \cA_\alpha)$ in $X$ as follows: Let $\alpha_1, \dots, \alpha_n$ be a collection of curves in $X$. Isotope each $\alpha_i$ so that it is embedded in $\widetilde{Y}.$ In Section \ref{Section: Homology Action}, we constructed a decorated surface $\cF_{\alpha_i}\sub \widetilde{Y}\times [-1,1]$ which realized the action of $\alpha_i$. Define $\cF_\alpha$ to be the decorated surface obtained stacking these surfaces on top of one another, i.e. $\cF_\alpha = \cF_{\alpha_n} \cup \dots \cup \cF_{\alpha_1}$. Let $(X, \cF_X) = (\overline{W_1}, \overline{\cF\cap W_1}) \cup (\widetilde{Y} \times [-1,1], \cF_\alpha) \cup (W_1, \cF \cap W_1)$. \\

Let $\eta_i$ be the curve $S^1 \times \pt$ in the $i$th $S^1 \times S^3$ summand of $X$. By taking $\alpha_i$ to be $\eta_i$, we obtain a decorated surface $\cF_\eta$ in $X$, with the curves $\eta_1, \dots, \eta_n \sub \cA_\eta$. Apply Proposition \ref{prop: surgery preserves link cob map} to see that
\[
F_{X, \cF_\eta, \frs} \simeq F_{X(\eta_1, \dots, \eta_n), \cF_\eta(\eta_1, \dots, \eta_n), \frs(\eta_1, \dots, \eta_n)}.
\]
Doing surgery on $X$ along the curves $\eta_1, \dots, \eta_n$ yields $Y_0 \times [-1,1] \#^n S^4$ which is, of course, diffeomorphic to the product $Y_0\times [-1,1]$. Recall that $\cF_\eta$ was defined by tubing on tori decorated by parallel copies of $\eta_i$. Doing surgery on the curves $\eta_i$ in these tori yields spheres embedded in the $S^4$ summands (and are therefore nullhomologous.) Therefore, $(Y_0\times[-1,1],\cF_\eta(\eta_1, \dots, \eta_n))$ can be built from the Morse-trivial link cobordism $(Y_0\times[-1,1],\cF_{Y_0\times [-1,1]})$ by tubing on a collection of spheres. By Proposition \ref{prop: sphere tubing}, the link cobordism map does not detect tubing on nullhomologous spheres, and hence
\[
F_{X(\eta_1, \dots, \eta_n),\cF_\eta(\eta_1, \dots, \eta_n), \frs(\eta_1, \dots, \eta_n)} \simeq F_{Y_0\times [-1,1], \cF_{Y_0\times [-1,1]}, \frs(\eta_1, \dots, \eta_n)}.
\]

By Lemma \ref{lemma: Surgery on doubles}, $D(W)$ can be obtained from $X$ by doing surgery on the attaching curves $\gamma_1, \dots, \gamma_n$ of the 2-handles of $W$. Now take the $\alpha_i$ to be the attaching curves $\gamma_1, ..., \gamma_n$ for the 2-handles for $W$, and consider the decorated link cobordism $\cF_\gamma$ in $X$ which realizes the action of $\gamma_1, \dots, \gamma_n$. Just as above, Proposition \ref{prop: surgery preserves link cob map} shows:
\[
F_{X, \cF_\gamma, \frs} \simeq F_{X(\gamma_1, \dots, \gamma_n), \cF_\gamma(\gamma_1, \dots, \gamma_n), \frs(\gamma_1, \dots, \gamma_n)} \simeq F_{D(W), \cF_\gamma(\gamma_1, \dots, \gamma_n), \frs(\gamma_1, \dots, \gamma_n)}.
\]
Again, the surface $\cF_\gamma(\gamma_1, \dots, \gamma_n)$ is obtained by tubing on the spheres that arise by doing surgery on the tori in $\cF_\gamma$. Here is a geometric argument that these spheres are nullhomologous in $D(W)$. The torus $T_i$ corresponding to $\gamma_i$ was defined by isotoping $\gamma_i$ into $\partial_+ W_1$ and taking the boundary of a regular neighborhood of $\gamma_i$ in $\partial_+ W_1$. Hence, $[T_i] = 0 \in H_2(X)$. By attaching a thickened disk along a meridian of $\gamma_i$, we obtain cobordism from $T_i$ to the sphere $S_i$ obtained by surgery on $\gamma$. Therefore, $[S_i] = [T_i] = 0$ in $H_2(X)$. Since $S_i$ is disjoint from $\gamma_i$ it also represents a class in $H_2(X - N(\gamma_i))$. $X$ is obtained from $X-N(\gamma_i)$ by attaching a 3- and 4-handle, so the relative homology group $H_1(X, X-N(\gamma_i)) = 0$, implying the map induced by inclusion $H_2(X-N(\gamma_i)) \ra H_2(X)$ is injective. In particular, this means $[S_i]$ is trivial in $H_2(X-N(\gamma_i))$, and therefore trivial in $H_2(X(\gamma_i))$ as well. Applying this argument to each $\gamma_i$ shows that all spheres $S_i$ are nullhomologous in $D(W)$.

Therefore, we can again apply Proposition \ref{prop: sphere tubing} to see that 
\[
F_{D(W), \cF_\gamma(\gamma_1, \dots, \gamma_n), \frs(\gamma_1, \dots, \gamma_n)} \simeq F_{D(W), D(\cF), \frs(\gamma_1, \dots, \gamma_n)},
\]
where $D(\cF)$ is the decorated cobordism obtained by doubling $\cF$. Since $(W, \cF)$ was concordance Morse-trivial, $(D(W), D(\Sigma))$ is as well. \\

Altogether, we have shown that the maps induced by $(Y_0\times[-1,1], \cF_{Y_0\times[-1,1]})$ and $(D(W), D(\cF))$ are chain homotopic to the maps induced by the link cobordisms $(X, \cF_\eta)$ and $(X, \cF_\gamma)$ respectively. Hence, we simply need to show that
\[
F_{X, \cF_\eta, \frs} \simeq F_{X, \cF_\gamma, \frs}.
\]
By Lemma \ref{lemma: homology action}, we have that 
\[
F_{X, \cF_\eta, \frs}(\cdot) \simeq F_{X, \cF_X, \frs}([\eta_1]\wedge \dots \wedge [\eta_n] \otimes \cdot)
\]
and 
\[
F_{X, \cF_\gamma, \frs}(\cdot) \simeq F_{X, \cF_X, \frs}([\gamma_1]\wedge \dots \wedge [\gamma_n] \otimes \cdot),
\]
where $(X, \cF_X)$ is concordance Morse-trivial. We will show that
\[
F_{X, \cF_X, \frs}([\eta_1]\wedge \dots \wedge [\eta_n] \otimes \cdot) \simeq F_{X, \cF_X, \frs}([\gamma_1]\wedge \dots \wedge [\gamma_n] \otimes \cdot),
\]

This is effectively proven in \cite[Theorem 4.10]{DLVW_ribbon}, but we will recall their argument for the convenience of the reader. By \cite[Proposition 5.1]{DLVW_ribbon}, 
\[[\eta_1] \wedge \dots \wedge [\eta_n]=[\gamma_1]\wedge \dots \wedge [\gamma_n] \in \left(\Lambda^*H_1(X)/\tors/ \langle H_1(Y_0)/\tors\rangle \right) \otimes \F_2,\] where $\langle H_1(Y_0)/\tors\rangle$ is the ideal generated by elements of $H_1(Y_0)/\tors$. Therefore, $[\eta_1] \wedge \dots \wedge [\eta_n]$ and $[\gamma_1]\wedge \dots \wedge [\gamma_n]$ differ by an element of $\Lambda^n(H_1(X)/\tors)\cap \langle H_1(Y_0)/\tors \rangle \otimes \F_2.$ By the linearity of the link cobordism maps, it suffices to show that $F_{X, \cF_X, \frs}(\x\otimes \xi) = 0$ for any $\xi \in \Lambda^n(H_1(X)/\tors)\cap \langle H_1(Y_0)/\tors \rangle \otimes \F_2.$\\

$\Lambda^n(H_1(X)/\tors)\cap \langle H_1(Y_0)/\tors \rangle \otimes \F_2$ is generated by elements of the form $\omega \wedge \left( \bigwedge_{i \in S} [\eta_i]\right)$ where $\omega$ is a wedge of elements in $H_1(Y_0)/\tors$ and $S$ is a proper subset of $\{1, ..., n\}$. So, we can take $\xi$ to be of this form. Since $S$ is a proper subset of $\{1, ..., n\}$ there is some $\eta_j$ which does not appear as a factor of $\xi$. We can realize the map $F_{X, \cF_X, \frs}(\xi \otimes \cdot)$ as a link cobordism map $F_{X, \cF_\xi, \frs}.$ By our construction of $\cF_\xi$, we can assume that $\cF_\xi$ is disjoint from the $j$th $S^1 \times S^3$ summand, of which $\eta_j$ is the core. Therefore, $(X, \cF_\xi)$ can be decomposed as $(X - T_j, \cF_\xi - \cD)\cup (T_j, \cD)$, where $T_j = (S^1 \times S^3) - D^4$ is the $j$th $S^1 \times S^3$ summand, and $\cD$ is a disk decorated with a single dividing arc. Hence, the map $F_{X, \cF_\xi, \frs}$ factors through $F_{T_j, \cD, \frs|_{T_j}}$. 

But, the map $F_{T_j, \cD, \frs|_{T_j}}$ is trivial. This map can be computed as the following composition: first, a 0-handle with a birth disk (which is identified with $\cD$) is born, then a 1-handle is attached with both feet on the 0-handle, with a trivially embedded annulus followed by 3-handle, again with a trivial annulus. The 1- and 3-handle maps take $\x$ to $\x \otimes \theta^+$ and $\x \otimes \theta^+$ to zero respectively, where $\theta^+$ is the top graded generator of $\cCFL^-(S^1\times S^2, \U, \frt_0)$.\\

Therefore, it has been established that
\[
F_{Y_0\times [-1,1], \cF_{Y_0\times [-1,1]}, \frs(\eta_1, \dots, \eta_n)} \simeq F_{D(W), D(\cF), \frs(\gamma_1, \dots, \gamma_n)}.
\]
The map on the left is induced by a Morse-trivial link cobordism, and therefore induces the identity map on $\cHFL^-(Y_0, K_0, \frt_0)$, where $\frt_0$ is the restriction of $\frs(\eta_1, \dots, \eta_n)$ to $Y_0$. $\SpinC$-structures on $Y_0$ extend uniquely over $W$ and $D(W)$ since they are $\Z$-homology cobordisms. Therefore $\frs(\gamma_1, ..., \gamma_n)$ \emph{is} the unique $\SpinC$ structure extending $\frs$ over $D(W)$. Therefore, for any $\frs\in \SpinC(W)$, the map $F_{D(W), D(\cF), D(\frs)}$ induces the identity.
\end{proof}

Having completed the proof in the case that $(W, \Sigma)$ is concordance Morse-trivial, it suffices to show that we can always reduce to this case. Given any concordance $C \sub W$, there is a Morse function $f$ on $W$ with the property that $f|_C$ has no critical points (for instance, simply identify $C \cong S^1\times I$ and extend the projection $S^1 \times I \ra I$ to a Morse function on $W$.) Moreover, the extension can be chosen in such a way as to give $W$ the structure of a ribbon cobordism. The following is well known; see for example, \cite[Chapter 6]{GS_4mflds} or \cite{millerzemke_stronglyhomotopy}.

\begin{prop}[\cite{GS_4mflds,millerzemke_stronglyhomotopy}]\label{prop: ribbon disks fusion number}
Let $K_0$ and $K_1$ be ribbon concordant in a homology cobordism $W$. Then, there is handle decomposition of the pair $(W, \Sigma)$ with the property that $W$ is ribbon and the handle decomposition for $\Sigma$ is trivial.
\end{prop}

\begin{figure}[h] 
     \centering
          \includegraphics[scale=.15]{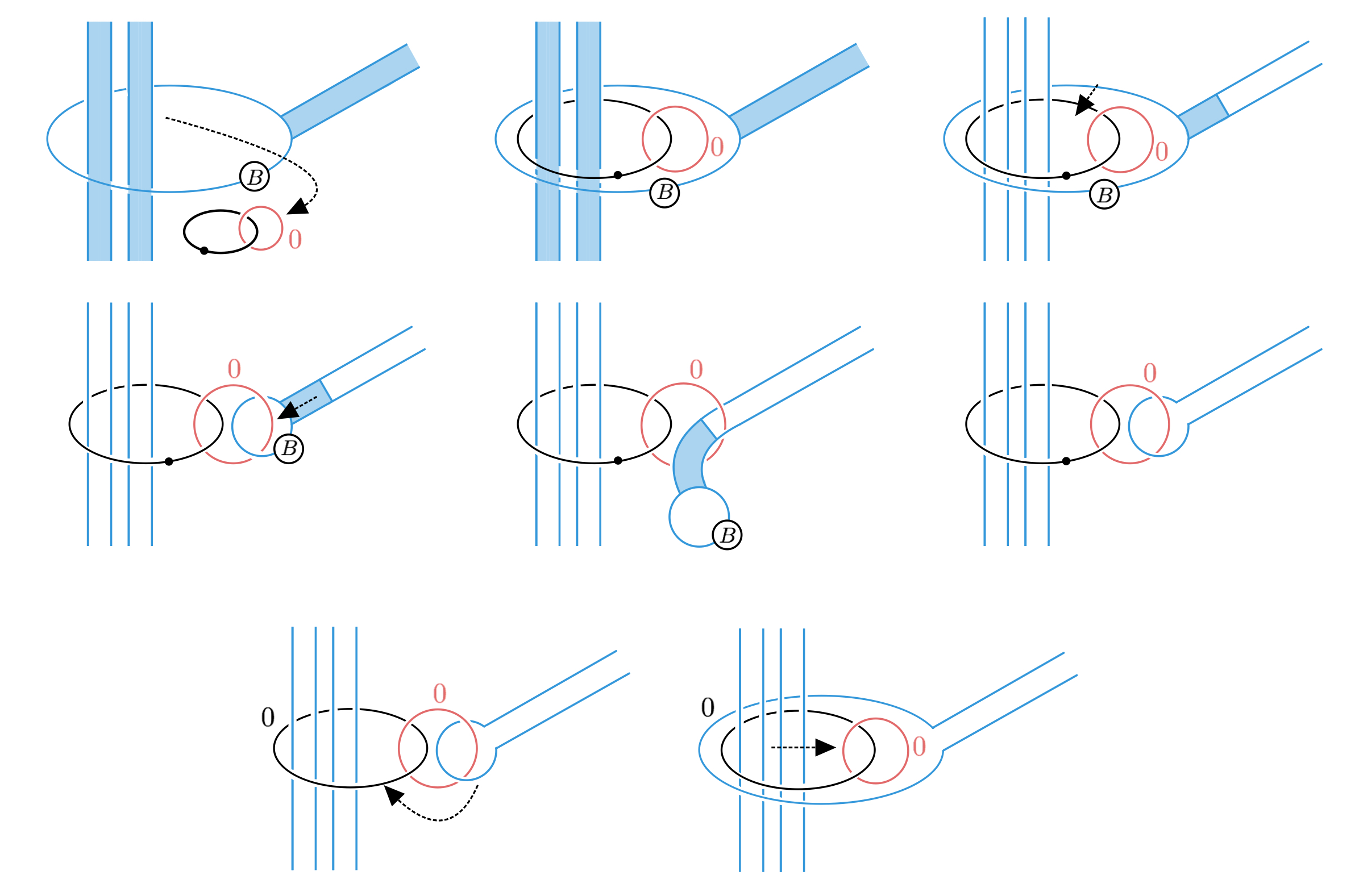}
               \caption{A procedure for trading handles of a concordance for handles of the ambient manifold.}
                    \label{fig: trading critical points}
\end{figure}

\begin{proof}[Proof of Theorem 1.] This now follows immediately from the previous propositions. For any ribbon $\Z$-homology concordance $(W, \cF)$, the doubled link cobordism induces the same map as a concordance Morse-trivial link cobordism $(D(W), \cF')$, and this map induces the identity map on $\cHFL^-(Y_0, K_0, \frs|_{Y_0})$. As $\frs$ extends uniquely over $D(W)$, the composition law implies that $F_{W, \cF, \frs}$ has a left inverse, namely $F_{\overline{W}, \overline{\cF}, \frs}$. Hence, $F_{W, \cF, \frs}$ is a split injection.
\end{proof}

\section{Torsion and Link Floer Homology}\label{section: torsion and link floer}

Let $\CFL^-(Y, K, \frs)$ be the $\F_2[V]$-module obtained from $\cCFL^-(Y, K, \frs)$ by setting $U = 0$ with differential {}
\[
\partial(\x) = \sum_{\y\in \T_\al\cap \T_\be}\sum_{\substack{\phi \in \pi_2(\x, \y),\\ \mu(\phi) =1, \\ n_\w(\phi) = 0}} \# \widehat{\mathcal{M}}(\phi)V^{n_\z(\phi)}\y.
\]
Let $\HFL^-(Y, K, \frs)$ be the homology of this complex.

A key property of the link Floer TQFT which is utilized in \cite{JMZ_TorsionOrder} is the following.

\begin{lemma}\cite[Lemma 3.1]{JMZ_TorsionOrder}\label{lemma: adding a tube}
Let $(W, \cF)$ be a decorated link cobordism. Let $\cF_V$ be a link cobordism obtained by adding a tube to the $\Sigma_\z$ region. Then, 
\[
F_{W, \cF_V, \frs} \simeq V\cdot F_{W, \cF, \frs}.
\] 
\end{lemma}
\begin{proof}
Choose a neighborhood the tube diffeomorphic to the 4-ball. By the composition law, it suffices to show that 
\[
F_{B^4, \cF_V \cap B^4} \simeq V\cdot F_{B^4, \cF},
\]
where $\cF$ is a pair of disks which bound a 2-component unlink in $\partial B^4$ decorated by a dividing arc. $\cF_V \cap B^4$ is obtained from $F_{B^4, \cF}$ by adding a tube connecting the two disk with feet in the $\z$-region. This map can be computed as the composition of two $\z$-band maps. This computation is carried out in \cite[Section 8.2]{zemke_linkcob}, and the resulting map is multiplication by $V$.
\end{proof}

With this tool at our disposal, we can prove an analogue of \cite[Proposition 4.1]{JMZ_TorsionOrder}.

\begin{prop}\label{prop: computation of ribbon link cobordisms}
Let $(W, \Sigma): (Y_0, K_0) \ra (Y_1, K_1)$ be a $\Z$-homology link cobordism. Let $h: W \ra \R$ be a Morse function compatible with $\Sigma$ with respect to which $W$ is ribbon. Suppose that $h|_\Sigma$ has $m$ critical points of index 0, $b$ critical points of index 1, and $M$ critical points of index 2. Let $\cF$ be a decoration of $\Sigma$ such that $\Sigma_\w$ is a regular neighborhood of an arc from $K_0$ to $K_1$. Then, 
\[
V^M\cdot F_{D(W), D(\cF), D(\frs)} \simeq V^{b-m}\cdot \id_{\cHFL^-(Y_0, K_0, \frs|_{Y_0})}.
\]
\end{prop}

\begin{proof}
The Morse function $h$ induces a movie presentation for $(W, \Sigma)$.

\begin{enumerate}
    \item $m$ birth disks appear disjoint from $K_0$, with boundaries $U_1, \dots, U_m$.
    \item $n$ 4-dimensional 1-handles are attached whose feet are disjoint from $\Sigma$.
    \item $m$ fusion bands $B_1$, ..., $B_n$ are attached which connect $K_0$ and $U_1, \dots, U_m$. After some band slides, the band $B_i$ has one foot in $U_i$ and the other in $K_0$.
    \item $b-m$ additional bands $B_{m+1}, ..., B_b$ are attached. 
    \item $n$ four-dimensional 2-handles are attached along curves $\gamma_1, \dots, \gamma_n $.
    \item $M$ death disks appear capping off unknotted components $U_1, \dots, U_M$ in the link obtained by doing band surgery.
\end{enumerate}
By playing this movie forward, and then again in reverse, we obtain a movie for $(D(W), D(\cF)).$ 

Consider the link cobordism which is obtained by deleting steps (5)-(8), i.e. remove the 2-handles, the deaths, the dual births, and the dual 2-handles. The resulting four-manifold $X$ is the double of the cobordism $W_1 = (Y_1\times[0,1])\cup \text{1-handles}$. The resulting surface is $D(\cF\cap W_1)$. Let $\cG_\gamma$ be the surface obtained from $D(\cF\cap W_1)$ by tubing on tori $T_1, ..., T_n$ which are the boundaries of regular neighborhoods of the $\gamma_i$ curves in $\partial_+ W_1$ as in the proof of Proposition \ref{prop: trivial concordance in W}. Surgery on the curves $\gamma_1, \dots, \gamma_n$ yields a link cobordism $(D(W), \cG)$. The decorated cobordism $\cG$ can be obtained from $D(\cF)$ by attaching $M$ tubes from the deaths disks to the dual birth disks and tubing on $n$ nullhomologous spheres which are the result of surgery on the tori $T_i$. We can arrange for the feet of the tubes to be sit in the subsurface $\Sigma_\z$. Attaching the nullhomologous spheres has no effect on the link cobordism map, and attaching the tubes has the result of multiplying by $V^M$ by Lemma \ref{lemma: adding a tube}. Therefore, by Proposition \ref{prop: surgery preserves link cob map},
\[
F_{X, \cG_\gamma, \frt} \simeq V^M\cdot F_{D(W), D(\cF), D(\frs)},
\]
where $\frt$ is determined by the fact that $\frt(\gamma_1, \dots, \gamma_n) = D(\frs)$.

Now consider the cobordism obtained by deleting steps (4)-(9) and again tubing on the tori corresponding to the $\gamma_i$ curves. The ambient four-manifold is still $X$, but removing steps (4) and (9) has the effect of removing the bands $B_{m+1}, ..., B_b$ and their duals from $\cG_\gamma$. Call this surface $\cH_\gamma$. Since the bands $B_{m+1}, ..., B_b$ and their duals form a collection of $b-m$ tubes, another application of Lemma \ref{lemma: adding a tube} shows
\[
F_{X, \cG_\gamma, \frt}\simeq V^{b-m}\cdot F_{X, \cH_\gamma, \frt}.
\]
But, surgery on $(X, \cH_\gamma)$ along $\gamma_1, \dots, \gamma_n$ yields the link cobordism $(D(W), \cH)$ which is the double of a ribbon homology concordance. So by Proposition \ref{prop: trivial concordance in W}
\[
F_{X, \cH_\gamma, \frt} \simeq \id_{\cCFL^-(Y_0, K_0, \frs|_{Y_0})}.
\]
Altogether then, we have that
\[
V^M\cdot F_{D(W), D(\cF), D(\frs)} \simeq V^{b-m}\cdot \id_{\cHFL^-(Y_0, K_0, \frs|_{Y_0})},
\]
as desired.
\end{proof}

\begin{proof}[Proof of Theorem \ref{theorem: tors bounds}]
This now follows by an argument identical to that of \cite[Theorem 1.2]{JMZ_TorsionOrder}.
\end{proof}

Unlike the inequality of \cite[Theorem 1.2]{JMZ_TorsionOrder}, this does \emph{not} give the symmetric result 
\[
\Ord_V(Y_1, K_1,\frs|_{Y_1}) \le \max\{m, \Ord_V(Y_0, K_0,\frs|_{Y_0})\} + 2g(\Sigma),
\] 
since $\overline{W}$ is not ribbon as we have defined it, unless $W = Y_0 \times [0, 1]$. 

\section{Applications}\label{section: applications}

We do have some immediate applications. Theorem \ref{thm: torsion bounds} gives a clear relationship between the torsion orders of ribbon homology cobordant knots. 

\begin{cor}
Let $(W, \Sigma, \frs): (Y_0, K_0) \ra (Y_1, K_1)$ be a ribbon $\Z$-homology cobordism, then
\[
\Ord_V(Y_0, K_0,\frs|_{Y_0}) - \Ord_V(Y_1, K_1,\frs|_{Y_1}) \le 2g(\Sigma).
\]
\end{cor}

Recall that the \textit{fusion number} $\cF us(K)$ of a ribbon knot $K$ in $S^3$ is the minimal number of bands in a handle decomposition of ribbon concordance $C$ from the unknot $U$ to $K$ in $S^3\times [0, 1]$.  By \cite{JMZ_TorsionOrder}, the torsion order of $K$ in $S^3$ provides a lower bound for the fusion number of $K$. There are a few possible generalizations we will consider.  

Let $K$ be a knot in a 3-manifold $Y$. Suppose that that $K$ is ribbon in $Y$, in the sense that there is a concordance $(Y\times [0, 1], C): (Y\times \{0\}, U) \ra (Y\times \{1\}, K)$ where $U$ is the boundary of a disk in $Y$ and $C$ is an annulus which is ribbon with respect to the projection $Y \times [0, 1] \ra [0, 1]$.

\begin{defn}
 We define the \textit{fusion number of $K$ in $Y$}, which we denote $\cF us_Y(K)$, to be the minimal number of bands in a ribbon concordance from $U$ to $K$ in $Y\times [0, 1]$. 
\end{defn}

\begin{corollary}
If $K$ is ribbon in $Y$, then
\[
\Ord_V(Y, K, \frs) \le \cF us_Y(K).
\]
\end{corollary}
\begin{proof}
Let $(Y\times [0, 1], C): (Y, U) \ra (Y, K)$ be a ribbon concordance with $b = \cF us_Y(K)$ bands (and therefore $b$ local minima as well). Theorem 2 then implies 
\[
\Ord_V(Y, K, \frs) \le \max\{b, \Ord_V(Y, U, \frs)\} = b,
\]
since $\HFL^-(Y, U, \frs)$ is torsion free.
\end{proof}

In another direction, one could also consider ribbon concordances in $\Z$-homology cobordisms. Given a $\Z$-homology concordance $(W, \Sigma): (Y_0, U) \ra (Y_1, K)$, one can always find a Morse function $h$ on $W$ compatible with $\Sigma$ so that $h|_\Sigma$ is ribbon, so we will continue to require that the ambient manifold is ribbon as well. However, by imposing the condition that the ambient manifold is ribbon, we have introduced an asymmetry which makes generalizing the fusion number to ribbon $\Z$-homology concordances somewhat subtle; in $S^3\times [0,1]$, concordances \emph{from} the unknot with no local maxima can be turned around and viewed as concordances \emph{to} the unknot with no local minima. However, this is clearly not the case for a ribbon homology concordance $(W, \Sigma): (Y_0, U) \ra (Y_1, K)$, as $\overline{W}$ is not ribbon. 

Therefore, since ribbon homology link cobordisms to and from the unknot differ, we can consider both cases: on the one hand, we have ribbon $\Z$-homology concordances $(W, \Sigma): (Y', U) \ra (Y, K)$ (where $W$ is a ribbon $\Z$-homology cobordism and $\Sigma$ is an annulus with \emph{no local maxima}), and on the other, link cobordisms $(W, \Sigma): (Y, K) \ra (Y', U)$ where $W$ is a ribbon $\Z$-homology cobordism and $\Sigma$ is an annulus with \emph{no local minima}. 

For the latter notion, Theorem \ref{theorem: tors bounds} immediately implies the following. 

\begin{cor}\label{cor: fusion bound}
Let $K$ be a knot in a 3-manifold $Y$. If $(W, \Sigma): (Y, K) \ra (Y', U)$ is a link cobordism such that $W$ is a ribbon $\Z$-homology cobordism and $\Sigma$ is an annulus with no local minima and $b$ index 1 critical points, then
\[
\Ord_V(Y, K, \frs)\le b,
\]
for any $\frs \in \SpinC(Y).$
\end{cor}
\
Let $\cF us^{\wedge}(Y, K)$ be minimal number of bands over all link cobordisms of the form $(W, \Sigma): (Y, K) \ra (Y', U)$ where $W$ is a ribbon $\Z$-homology cobordism and $\Sigma$ is an annulus with no local minima. The previous result, of course, implies that $\Ord_V(Y, K, \frs) \le \cF us^{\wedge}(Y, K)$.\\

When considering $\Z$-homology concordances $(W, \Sigma): (Y, U) \ra (Y', K)$, some care is needed in defining a sensible notion of fusion number, since  handles can be traded between the surface and ambient 4-manifold (cf. Proposition \ref{prop: ribbon disks fusion number}). In particular, $\Ord_V(S^3, K)$ cannot possibly be a lower bound on the number of bands required in such link cobordisms.

It seems more sensible to consider the minimal number of 2-handles in the complement of such a concordance, i.e. the \emph{strong homotopy fusion number}, though the torsion order cannot be a lower bound for this quantity by \cite{HKP_ribbon_cables}; there Hom-Kang-Park produce knots in $S^3$ with homotopy fusion number 1 but with arbitrarily large torsion order. 

\bibliographystyle{alpha}{}
\bibliography{mathbib}
\end{document}